\documentclass[12pt]{amsart}
\usepackage{amsthm,amsfonts,amssymb,amscd,mathrsfs,
latexsym,amsmath
}
\usepackage{bezier,longtable,amstext}
\usepackage{amssymb}

\usepackage[all]{xy}

\newcommand{\Pic}{\operatorname{Pic}\nolimits}

\renewcommand{\phi}{\varphi}

\newtheorem{theorem}{Theorem}[section]
\newtheorem{proposition}[theorem]{Proposition}

\newtheorem{lemma}[theorem]{Lemma}
\newtheorem{corollary}[theorem]{Corollary}

\theoremstyle{definition}
\newtheorem{definition}[theorem]{Definition}
\newtheorem{example}[theorem]{Example}
\newtheorem{proposition-definition}[theorem]{Proposition-Definition}

\nonstopmode

\begin{document}

\title{An algebraic-geometric construction of ind-varieties of
generalized flags}

\author[I.Penkov]{\;Ivan~Penkov}

\address{
Jacobs University Bremen \\
Campus Ring 1, 28759 
Bremen, Germany}
\email{i.penkov@jacobs-university.de}

\author[Tikhomirov]{\;Alexander S.~Tikhomirov}

\address{
Faculty of Mathematics\\
National Research University
Higher School of Economics\\
6 Usacheva str.\\
119048 Moscow, Russia}
\email{astikhomirov@mail.ru}

\thispagestyle{empty}

\begin{abstract}
We define the class of admissible linear embeddings of flag
varieties. The definition is given in the general language of
algebraic geometry. We then prove that an admissible linear
embedding of flag varieties has a certain explicit form in
terms of linear algebra. This result enables us to show that
any direct limit of admissible embeddings of flag varieties is
isomorphic to an ind-variety of generalized flags as defined in
\cite{DP}. These latter ind-varieties have been introduced in
terms of the ind-group $SL(\infty)$ (respectively, $O(\infty)$
or $Sp(\infty)$ for isotropic generalized flags), and the 
current paper constructs them in purely algebraic-geometric 
terms.



\end{abstract}

\maketitle

\section{Introduction}\label{sec1}
\label{Introduction}

Flag varieties play a fundamental role in geometry, and so do 
their analogues in ind-geometry. In this paper, we would like 
to place these analogues under the looking glass and provide a 
new characterization of the ind-varieties of generalized flags
constructed in \cite{DP}. Around 20 years ago, I. Dimitrov 
and the first author realized that in the context of 
ind-geometry the notion of a flag of vector subspaces in an
ambient infinite-dimensional vector space is rather 
subtle. More precisely, in addition to the obvious three 
types of infinite flags, that is, chains of vector subspaces 
enumerated by $\mathbb{Z}_{>0}$, $\mathbb{Z}_{<0}$ or  
$\mathbb{Z}$, there is the need to consider chains of 
subspaces enumerated by more general totally ordered sets in 
which every element has an immediate predecessor or an 
immediate successor, but possibly not both. Such chains, 
satisfying the additional
condition that every vector of the ambient vector space is 
contained in some space of the chain but not in its immediate 
predecessor, were christened \textit{generalized flags} in 
\cite{DP}. The main result of \cite{DP} can be summarized 
roughly as follows: generalized flags in a 
countable-dimensional vector space are in a natural 1-1 
correspondence with splitting parabolic subgroups $P$ of the 
ind-group $GL(\infty)$, and hence the points of homogeneous 
ind-spaces of the form $GL(\infty)/P$ can be thought of as 
generalized flags. A similar statement about isotropic 
generalized flags holds for the ind-groups $O(\infty)$ and 
$Sp(\infty)$. In particular, the concept of generalized flag, 
and therefore also the notion of an ind-variety of 
generalized flags, has been motivated in the past by the 
notion of a parabolic subgroup of an ind-group like 
$GL(\infty)$, $O(\infty)$, $Sp(\infty)$. 

The main purpose of the present paper is to propose another,
purely algebraic-geometric, approach to the ind-varieties of 
generalized flags. More precisely, we define 
\textit{admissible linear embeddings} of usual flag varieties
\begin{equation}\label{emb 1}
Fl(m_1,...,m_k,V)\hookrightarrow Fl(n_1,...,n_{\tilde{k}},V')
\end{equation}
and show that an ind-variety obtained as a direct limit of
such linear embeddings is isomorphic to an ind-variety of 
generalized flags. In particular, such a linear direct limit 
is automatically a homogeneous ind-space of $GL(\infty)$. We 
also consider isotropic generalized flags and prove a similar 
result for the ind-groups $O(\infty)$ and $Sp(\infty)$.
In this way, the notion of an admissible linear embedding of 
flag varieties leads naturally to the concept of generalized
flag. A small part of this program has already been 
carried in our paper \cite{PT} where we characterize linear 
embeddings of grassmannians, and then as a consequence 
describe linear ind-grassmannians up to isomorphism.

Our main new result concerning embeddings of 
finite-dimensional flag varieties is finding an explicit form
of a class of embeddings \eqref{emb 1} which we call 
\textit{admissible}. We define an admissible linear embedding
in general algebraic-geometric terms, and then show that such 
an embedding is nothing but an extension of a flag from
$Fl(m_1,...,m_k,V)$ to a possibly longer flag in
$Fl(n_1,...,n_{\tilde{k}},V')$, given by an explicit formula
from linear algebra. We call the latter embeddings 
\textit{standard extensions}. This enables us to prove that a
direct limit of admissible linear embeddings is isomorphic to
an ind-variety of generalized flags as in \cite{DP}, as it is
relatively straightforward to show that direct limits of
standard extensions have this property. 

The paper is concluded by an appendix in which we present two 
examples of direct limits of linear but non-admissible 
embeddings of flag varieties, that are not isomorphic to 
ind-varieties of generalized flags.

{\bf Acknowledgements.} I.P. thanks Vera Serganova for a useful
discussion, which took place several years ago, on the general 
idea of an algebraic-geometric approach to ind-varieties of 
generalized flags. I.P. was supported in part by DFG-grant PE 
980/7-1. A.S.T. thanks the Max Planck Institute for Mathematics in Bonn, where this work was partially done during the winter of 2017, for hospitality and financial support.

{\bf Notation.} The sign $\subset$ stands for not necessarily 
strict set-theoretic inclusion. By $G(m,V)$ we denote the 
grassmannian of $m$-dimensional subspaces of $V$ for $1\le 
m\le\dim V$. We also use the notation $\mathbb{P}(V)$ for 
$G(1,V)$. If $a:X\to Y$ is a morphism of algebraic varieties,
by $a^*$ and $a_*$ we denote respectively the pullback or 
pushforward of vector bundles. The superscript $(\cdot)^{\vee}$
indicates dual space or dual vector bundle.

\vspace{1cm}
\section{Definition of linear embedding of flag varieties} 
\label{preliminaries}
\vspace{5mm}

In this section we give the basic definitions of linear 
embeddings of flag varieties including the case of isotropic 
flag varieties.

The base field is $\mathbb{C}$ and all vector spaces, 
varieties and ind-varieties considered below are defined over
$\mathbb{C}$. Let $V$ be a vector space of dimension $\dim 
V\ge2$. For any increasing sequence of positive integers 
$1\le m_1<...< m_k<\dim V$, we consider the \textit{flag 
variety} $Fl(m_1,...,m_k,V):=\{(V_{m_1},...,V_{m_k})\in 
G(m_1,V)\times...\times G(m_k,V)\ |\ V_{m_1}\subset...\subset 
V_{m_k}\}$. We denote its points by $F=(0\subset V_{m_1}
\subset...\subset V_{m_k}\subset V)$ or sometimes by 
$F=(V_{m_1}\subset...\subset V_{m_k})$. The ordered $k$-tuple
$(m_1,...,m_k)$ is the \textit{type} of a flag $F\in Fl(m_1,...,m_k,V)$.

There is a natural embedding 
$$
j:\ Fl(m_1,...,m_k,V)\hookrightarrow G(m_1,V)\times...\times 
G(m_k,V)
$$ 
and there are projections 
$$
\pi_i:\ Fl(m_1,...,m_k,V)\to G(m_i,V),\ F=(V_{m_1}\subset 
...\subset V_{m_k})\mapsto V_{m_i},\ i=1,...,k.
$$ 
We have 
$$
{\rm{Pic}}~Fl(m_1,...,m_k,V)=\mathbb{Z}[L_1]\oplus...\oplus
\mathbb{Z}[L_k],
$$
where
\begin{equation*}\label{generators}
L_i:=\pi_i^*\mathcal{O}_{G(m_i,V)}(1),\ \ \ i=1,...,k.
\end{equation*}
Here, $\mathcal{O}_{G(m_i,V)}(1)$ denotes the invertible sheaf 
on $G(m_i,V)$ satisfying $H^0(\mathcal{O}_{G(m_i,V)}(1))=
\wedge^{m_i}(V^*)$. By definition, $[L_1],...,[L_k]$ is a 
\textit{preferred set of generators} of ${\rm{Pic}}~Fl(m_1,...,
m_k,V)$.

\vspace{4mm}

Let $V$ be equipped with a non-degenerate symmetric bilinear 
form on $V$. For our purposes, we can assume that $\dim 
V\ge7$. For $1\le k\le[\frac{\dim V}{2}]$, the 
\textit{orthogonal grassmannian} $GO(m,V)$ is defined as the
subvariety of $G(m,V)$ consisting of isotropic 
$m$-dimensional subspaces of $V$. Unless $\dim V=2m$, 
the variety $GO(m,V)$ is a smooth irreducible variety. For 
$\dim V=2m$, the orthogonal grassmannian is a disjoint 
union of two isomorphic smooth irreducible components, and 
they are both isomorphic to $GO(m-1,V')$ where $\dim V'=2m-1$. 
Slightly abusing notation, we will denote by $GO(m,V)$ each of 
these two components. 

If $m\ne\frac{\dim V}{2}-1$, then $\Pic GO(m,V)=\mathbb{Z}
[\mathcal{O}_{GO(m,V)}(1)]$, where the sheaf $\mathcal{O}_{GO
(m,V)}(1)$ posesses the following property: if $t: GO(m,V)
\hookrightarrow G(m,V)$ is the tautological embedding, then
\begin{equation*}\label{restrictn Pic}
t^*\mathcal{O}_{G(m,V)}(1)\cong\Bigl\{
\begin{array}{lll}
\mathcal{O}_{GO(m,V)}(1) & \mathrm{for} & 
m\ne\frac{\dim V}{2},\\
\mathcal{O}_{GO(k,V)}(2) &  \mathrm{for} & 
m=\frac{\dim V}{2}.
\end{array}
\end{equation*}
If $m=\frac{\dim V}{2}-1$, then for any $V_{m-1}\in GO(m-1,V)$ 
there is a unique $V_m\in GO(m,V)$ such that $V_{m-1}\subset 
V_m$. Thus there is a well-defined morphism
\begin{equation}\label{flag ortho n-1,n}
\theta:\ GO(m-1,V)\to GO(m,V),\ \ \ V_{m-1}\mapsto V_m,
\ \ \ \textrm{where}\ \ \ V_m\supset V_{m-1}. 
\end{equation}
Consequently,
\begin{equation*}\label{Pic GO special}
\Pic GO(m-1,V)=\mathbb{Z}[\theta^*\mathcal{O}_{GO(m,V)}(1)]
\oplus\mathbb{Z}[\mathcal{O}_{GO(m-1,V)}(1)],
\end{equation*}
where by $\mathcal{O}_{GO(m-1,V)}(1)$ we denote the 
$\theta$-relatively ample Grothendieck sheaf determined by the 
property that $\theta_*\mathcal{O}_{GO(m-1,V)}(1)$ is the 
universal quotient bundle on $GO(m,V)$.

\vspace{2mm}

Next, let $1\le m_1<...< m_k$ be an increasing sequence of 
positive integers, where $m_k\le [\frac{\dim V}{2}]$. The \textit{orthogonal flag variety} 
$FlO(m_1,...,m_k,V)$ is defined as
$$
FlO(m_1,...,m_k,V):=\{(V_{m_1},...,V_{m_k})\ |\ V_{m_1}\in GO(m_i,V),\ V_{m_1}\subset...\subset V_{m_k}\},
$$ 
where, according to our convention, we assume $GO(m_k,V)$
connected if $m_k=\frac{\dim V}{2}$.
Similarly to the case of usual flag varieties, there is a 
natural embedding $j:\ FlO(m_1,...,m_k,V)
\hookrightarrow GO(m_1,V)\times...\times GO(m_k,V)$ and there 
are projections $\pi_i:\ Fl(m_1,...,m_k,V)\to GO(m_i,V),\ 
(V_{m_1}\subset...\subset V_{m_k})\mapsto V_{m_i},\ 
i=1,...,k$. Unless $m_k=\frac{\dim V}{2}-1,$ we have 
\begin{equation*}\label{Pic FlO}
{\rm{Pic}}~FlO(m_1,...,m_k,V)=\mathbb{Z}[L_1]\oplus...\oplus
\mathbb{Z}[L_k],
\end{equation*}
where
\begin{equation*}\label{generators ortho}
L_i:=\pi_i^*\mathcal{O}_{GO(m_i,V)}(1),\ \ \ i=1,...,k.
\end{equation*}
The isomorphism classes $[L_i]$ are a \textit{preferred set 
of generators} of ${\rm{Pic}}~FlO(m_1,...,m_k,V)$. 
If $m_k=\frac{\dim V}{2}-1,$ then there is an additional
preferred generator $[(\theta\circ\pi_{k-1})^*
\mathcal{O}_{GO(m_k+1,V)}(1)]$ of ${\rm{Pic}}FlO(m_1,...,m_k,
V)$.
\vspace{4mm}

Let now $V$ be equipped with a non-degenerate symplectic 
form. This implies that $\dim V\in2\mathbb{Z}_{>0}$. Assume 
$1\le m\le \frac{1}{2}\dim V$. By definition, the 
\textit{$m$-th symplectic grassmannian} $GS(m,V)$ is the 
smooth irreducible subvariety of $G(m,V)$ consisting of 
isotropic $m$-dimensional subspaces of $V$. It is known that 
\begin{equation*}\label{Pic GS}
\Pic GS(m,V)=\mathbb{Z}[\mathcal{O}_{GS(k,V)}(1)],\ \ \  
\mathcal{O}_{GS(k,V)}(1)=i^*\mathcal{O}_{G(k,V)}(1),
\end{equation*}
where $i:GS(m,V)\hookrightarrow G(m,V)$ is the tautological 
embedding. For a fixed increasing sequence of positive integers $1\le m_1<...\le m_k\le\frac{\dim V}{2}$, the 
\textit{symplectic flag variety} is defined as
$$
FlS(m_1,...,m_k,V):=\{(V_{m_1},...,V_{m_k})\in 
GS(m_1,V)\times...\times GS(m_k,V)\ |\ V_{m_1}\subset...
\subset V_{m_k}\}.
$$ 
We have a natural embedding $j:\ FlS(m_1,...,m_k,V)
\hookrightarrow GS(m_1,V)\times...\times GS(m_k,V)$ and 
projections $\pi_i:\ Fl(m_1,...,m_k,V)\to GS(m_i,V),\ (V_{m_1}
\subset ...\subset V_{m_k})\mapsto V_{m_i},\ i=1,...,k,$. 
Moreover,
\begin{equation*}\label{Pic FlS}
{\rm{Pic}}~FlS(m_1,...,m_k,V)=\mathbb{Z}[L_1]\oplus...\oplus
\mathbb{Z}[L_k],
\end{equation*}
where
\begin{equation*}\label{generators sympl}
L_i:=\pi_i^*\mathcal{O}_{GS(m_i,V)}(1),\ \ \ i=1,...,k,
\end{equation*}
The isomorphism classes $[L_i]$ are a \textit{preferred set of 
generators} of ${\rm{Pic}}~FlS(m_1,...,m_k,V)$.

We now proceed to the definition of linear embeddings of flag 
varieties and their orthogonal and symplectic analogues.

\begin{definition}\label{lin emb isotr flags}
Let $k$ and $\tilde{k}$ be positive integers with $1<k\le
\tilde{k}$. An embedding of flag varieties 
\begin{equation*}\label{usual embedding}
\varphi:\ X\hookrightarrow Y,
\end{equation*}
where $X=Fl(m_1,...,m_k,V),\ Y=Fl(n_1,...,n_{\tilde{k}},V')$,
or $X=FlO(m_1,...,m_k,V),\ Y=FlO(n_1,...,n_{\tilde{k}},V')$,
or $X=FlS(m_1,...,m_k,V),\ Y=FlS(n_1,...,n_{\tilde{k}},V')$,
is a \textit{linear embedding} if, for any $j$, $1\le j\le\tilde{k}$, we have
\begin{equation*}\label{linearity}
[\varphi^*M_j]=0\ \ \ \textrm{or}\ \ \  [\varphi^*M_j]=[L_i]
\end{equation*}
for some $i$, $1\le i\le k$, where $[L_1],...,[L_k]$ and 
$[M_1],...,[M_{\tilde{k}}]$ are the preferred sets of 
generators of ${\rm{Pic}}X$ and ${\rm{Pic}}Y$.
\end{definition}

\begin{example}\label{lin emb isotr grass}
Assume that $k=\tilde{k}=1$ in Definition \ref{lin emb
isotr flags}. Then $X$ and $Y$ are grassmannians, orthogonal 
grassmannians, or symplectic grassmannians. In all cases, 
except when $X=GO(m,V)$ and $Y=GO(n,V')$ for $(m,\dim V)=
(l-1,2l)$ or $(n,\dim V')=(r-1,2r)$, a linear embedding 
$\varphi:X\to X$ is simply an embedding with $\varphi^*[M]=
[L]$, where $[L]$ and $[M]$ are respective ample generators of the Picard groups ${\rm{Pic}}Y$ and ${\rm{Pic}}X$, 
cf. \cite[Def. 2.1]{PT}.

In the remaining cases, a linear embedding $\varphi:X\to Y$   
exists if and only if $X\simeq GO(l-1,V)$, $Y\simeq GO(r-1,V'
)$ for $l\le r$, and here the linearity of $\phi$ implies 
$\phi^*\mathcal{O}_{GO(r-1,V')}(1)\cong\mathcal{O}_{GO(l-1,V)}
(1)$, $\varphi^*\theta'^{*}\mathcal{O}_{GO(r,V')}(1)\cong
\theta^*\mathcal{O}_{GO(l,V)}(1)$, where $\theta:GO(l-1,V)\to 
GO(l,V)$ 
and $\theta': GO(r-1,V')\to GO(r,V')$ are the projections 
defined in \eqref{flag ortho n-1,n}. To see this, one has to
show (we leave this to the reader) that it is impossible to
have an embedding $\phi:\ GO(l-1,V)\to GO(r-1,V')$ with
$\varphi^*\theta'^{*}\mathcal{O}_{GO(r,V')}(1)\cong
\mathcal{O}_{GO(l-1,V)}(1)$, $\phi^*\mathcal{O}_{GO(r-1,V')}
(1)\cong\theta^*\mathcal{O}_{GO(l,V)}(1)$.

\end{example}

A linear embedding $\phi$ as in Definition \ref{lin emb 
isotr flags} induces a partition with $k+1$ parts 
$\{0,1,...,\tilde{k},\tilde{k}+1\}=I_0\sqcup I_1\sqcup I_2
\sqcup...\sqcup I_k$ such that $0\in I_0$ and $j\in I_0$ 
iff $\phi^*[M_j]=0$, respectively, $j\in I_i$ for $i\ge1$ 
iff $\phi^*[M_j]=[L_i]$. The map $j\mapsto i$, for $j\in I_i$, 
is a surjection which we denote by $p$. By definition, 
$p(0)=0$.

\begin{proposition}\label{extension to Grassm growth}
(i) Let $\varphi:\ Fl(m_1,...,m_k,V)\hookrightarrow Fl(n_1,...,
n_{\tilde{k}},V')$ be a linear embedding. Then $\phi$ induces a
collection of morphisms of grassmannians
\begin{equation*}\label{colln general}
\varphi_{[i]}=\{\varphi_{i,j}\}_{i=p(j)}:\ G(m_i,V)\to
\underset{j>0:p(j)=i}{\prod}G(n_j,V'),\ \ \ 0\le i\le k,
\end{equation*}
such that the diagram
\begin{equation}\label{diagram 1}
\xymatrix{
Fl(m_1,...,m_k,V)\ar@{^{(}->}[rrr]^-{\varphi}\ar@{^{(}->}[d]^-{
j} & & & Fl(n_1,...,n_{\tilde{k}},V')
\ar@{^{(}->}[d]^-{j'}\\
G_0\times G(m_1,V)\times...\times G(m_k,V) \ar@{^{(}->} 
[rrr]^-{\phi_{[1]}\times...\times\phi_{[k]}} & & & 
G(n_1,V')\times ...\times G(n_{\tilde{k}},V')}
\end{equation}
where $j$ and $j'$ are the natural embeddings, is commutative.
Here $G_0$ is a single point, and is present in the diagram if
and only if there are constant morphisms $\phi_{0=p(j),j}:
G_0\to G(n_j,V')$.\\
(ii) Similar statements hold in the orthogonal and symplectic 
cases.
\end{proposition}
In the proof, we will need the following.
\begin{lemma}\label{lemma1}
\textit{Let $X,\ Y,\ Z$ be projective varieties with
$Y$ smooth, and let $a:X\to Y$ and $b:X\to Z$ be morphisms
such that $a$ is surjective and $b$ is constant on the fibers 
of $a$. Then there exists  a morphism $f:Y\to Z$ such that 
$b=f\circ a$.} 
\end{lemma}
\begin{proof}
Consider the morphism $g:X\to Y\times 
Z,\ x\mapsto(a(x),b(x))$, and let $Y\xleftarrow{a'}Y\times Z
\xrightarrow{b'}Z$ be the projections onto factors  so that
$a=a'\circ g$ and $b=b'\circ g$. Since $b$ is constant on the
fibers of $p$, it follows that $\tilde{a}:=a'|_{g(X)}:
g(X)\to Y$ is a bijection. Therefore, as $Y$ is smooth, 
$\tilde{a}$ is an isomorphism (see, e.g., \cite[Ch.2, Section 
4.4, Thm. 2.16]{S}). The desired morphism $f$ is now the 
composition $f=b'\circ\tilde{a}^{-1}$.
\end{proof}

\textit{Proof of Proposition \ref{extension to Grassm growth}}.
(i) We consider the case $k=\tilde{k}=2$. For 
arbitrary $k,\ \tilde{k}$ the proof goes along the same lines,
and we leave the details to the reader. 
Set $[L_1]:=\phi^*[M_{j_1}],\ [L_2]:=\phi^*[M_{j_2}]$, and let 
$\pi_i:Fl(m_1,m_2,V)\to G(m_i,V),\ \pi'_i:Fl(n_1,n_2,V')\to 
G(n_i,V'),\ i=1,2,$ be the natural projections. For an 
arbitrary point $x=(x_1,x_2)=(V_{m_1},V_{m_2})\in Fl(m_1,m_2,V)
\subset G(m_1,V)\times G(m_2,V)$, consider the fibres 
$\pi_i^{-1}(x_i)\subset 
F,\ i=1,2,$ through the point $x$. Since $\varphi$ is a linear 
embedding, we have $M_{j_1}|_{\varphi
(\pi_1^{-1}(x_1))}\simeq\varphi^*M_{j_1}|_{\pi_1^{-1}(x_1)}
\simeq\mathcal{O}_{\pi_1^{-1}(x_1)}\simeq\mathcal{O}_{\varphi
(\pi_1^{-1}(x_1))}$. As $\varphi(\pi_1^{-1}(x_1))$ is an irreducible variety and $M_{j_1}={\pi'}_1^*
\mathcal{O}_{G(n_{j_1},V')}(1)$, where $\mathcal{O}_{G(n_{j_1},
V')}(1)$ is an ample sheaf, it follows from the above 
isomorphisms that $\pi'_{j_1}$ is constant on the 
variety $\varphi(\pi_1^{-1}(x_1))$. Equivalently, the 
morphism $\pi'_{j_1}\circ\varphi$ is constant on the fibres of 
the projection $\pi_1$.
%
% We now invoke the following general 
%fact about morphisms of projective varieties.
%\vspace{1mm}
%\textit{Claim. Let $X,\ Y,\ Z$ be projective manifolds with
%$Y$ smooth, and let $p:X\to Y$ and $q:X\to Z$ be morphisms
%such that $p$ is surjective and $q$ is constant on fibers of 
%$p$. Then there exists  a morphism $f:Y\to Z$ such that 
%$q=f\circ p$.} \\
%\textit{Proof of claim.} Consider the morphism $g:X\to Y\times 
%Z,\ x\mapsto(p(x),q(x))$ and let $Y\xleftarrow{p'}Y\times Z
%\xrightarrow{q'}$ be the projections onto factors, so that
%$p=p'\circ g$ and $q=q'\circ g$. Since $q$ is constant on 
%fibers of $p$, it follows that $\tilde{p}:=p'|_{\phi(X)}:
%\phi(X)\to Y$ is a bijection. Therefore, as $Y$ is smooth, 
%$\tilde{p}$ is an isomorphism (see, e.g., \cite[Ch.2, Section 
%4.4, Thm. 2.16]{S}). The desired morphism $f$ is now the 
%composition $f=q'\circ\tilde{p}^{-1}$. ~\hfill$\Box$

Lemma \ref{lemma1} implies that $\pi'_1\circ\varphi$ 
factors through the projection $\pi_1$, i.e. there is a 
well-defined morphism 
\begin{equation}\label{phi1}
\varphi_1:\ G(m_1,V)\to G(n_{j_1},V'),\ 
x_1\mapsto \pi'_{j_1}(\varphi(\pi_1^{-1}(x_1)))
\end{equation}
such that $\varphi_1\circ \pi_1=\pi'_{j_1}\circ
\varphi$. In a similar way there is a well-defined
morphism
\begin{equation}\label{phi2}
\varphi_2:\ G(m_2,V)\to G(n_{j_2},V'),\ 
x_2\mapsto \pi'_{j_2}(\varphi(\pi_1^{-1}(x_2)))
\end{equation}
such that $\varphi_2\circ p_2=\pi'_{j_2}\circ
\varphi$. By construction, $\phi_1$ and $\phi_2$ are linear 
morphisms.  

Considering now $Fl(m_1,m_2,V)$ and $Fl(n_1,n_2,V')$ as lying, 
respectively, in $G(m_1,V)\times G(m_2,V)$ and in 
$G(n_1,V')\times G(n_2,V')$, for any points $x=(x_1,x_2)\in 
Fl(m_1,m_2,V)$ and $x'=(x'_1,x'_2)\in Fl(n_1,n_2,V')$ we have
$$
x=\pi_1^{-1}(x_1)\cap \pi_2^{-1}(x_2),\ \ \ 
x'={\pi'}_1^{-1}(x'_1)\cap {\pi'}_2^{-1}(x'_2).
$$ 
This together with (\ref{phi1}) and (\ref{phi2}) shows 
that, if $x'_{j_i}=\varphi_i(x_i),\ i=1,2$, then
$$
\varphi(x)=\varphi(\pi_1^{-1}(x_1))\cap\varphi(\pi_2^{-1}(x_2))\in
{\pi'}_{j_1}^{-1}(x'_{j_1})\cap{\pi'}_{j_2}^{-1}(x'_{j_2})=
(\phi_1\times\phi_2)(x),
$$
i.e. the diagram (\ref{diagram 1}) is commutative for $k=2$. 

We leave to the reader to make (ii) precise and check that the 
above proof extends to this case.  ~\hfill$\Box$

\vspace{1cm}
\section{Standard extensions of flag varieties}
\label{linear embed}
\vspace{5mm}

In this section we introduce and study a class of 
embeddings of flag varieties that we call standard 
extensions. In almost all cases, standard extensions are 
linear embeddings in the sense of Section \ref{preliminaries}.

We start by considering the case of grassmannians. Let
\begin{equation}\label{def of strict}
\varphi: G(m,V) \hookrightarrow G(n,V')
\end{equation}
be a regular morphism. Assume $\dim V'>\dim V$, $m\ne0$, $m\ne
\dim V$. We say that $\varphi$ is a \textit{strict standard 
extension} if there exists an isomorphism of vector spaces 
$V'=V\oplus\widehat{W}$ and a subspace $W\subset\widehat{W}$, 
such that 
\begin{equation*}\label{def of strict2}
\varphi(V_m)=V_m\oplus W
\end{equation*} 
where $V_m\subset V$ is an arbitrary point of $G(m,V)$. If 
$m=0$ or $m=\dim V$, a morphism \eqref{def of strict} is
necessarily constant and we call it a \textit{constant strict 
standard extension}. In this case we set $W:=\phi(G(m,V))$.

It is easy to check that a nonconstant strict standard  
extension is a linear embedding. 

By a \textit{modified standard extension} we understand an 
embedding \eqref{def of strict} for which there exists a 
strict standard extension
$$
\varphi':G(m,V)\hookrightarrow G(\dim V'-n,V'^{\vee})
$$
such that $\varphi=d\circ\varphi'$ where
$$
d:G(\dim V'-n,V'^{\vee})\xrightarrow{\sim}G(n,V')
$$
is the duality isomorphism. In what follows, a \textit{standard
extension} will mean a strict standard extension or a modified
standard extension.    

Note that if a morphism \eqref{def of strict} is linear, it is
not necessarily a standard extension. For instance, the reader 
can prove that the Pl\"ucker embedding 
$$
\psi:\ G(m,V)\hookrightarrow G(1,\wedge^mV)=\mathbb{P}
(\wedge^m V)
$$
is a standard extension if and only if $m=1$ or $m=\dim V-1$.
On the other hand, the Pl\"ucker embedding is of course a 
linear embedding.

In the case of orthogonal and symplectic grassmannians, a 
strict standard extension is defined in the same way with the 
additional requirement that the decomposition $V'=V\oplus U$
be orthogonal and that the spaces $V_m$ and $W$ are isotropic.
In these cases there is no need to consider modified standard
extensions (as the spaces $V$ and $V^{\vee}$ are identified via
the respective non-degenerate form), and the terms strict 
standard extension and standard extension are synonyms.

Here is a definition of strict standard extension $\varphi$
of grassmannians which refers only to the data of linear 
algebra which can be recovered canonically from the embedding
$\varphi$.

\begin{definition}\label{strict}
Let $\dim V'>\dim V$.
A morphism of grassmannians $\varphi:G(m,V)\hookrightarrow 
G(n,V')$ is said to be a \textit{strict standard extension} 
if either $G(m,V)$ is a point (i.e. $m=0$ or $m=\dim V$, and 
$\phi$ is constant) or there exists a subspace $U\subset V'$ 
and a surjective linear operator $\varepsilon:\ 
U\twoheadrightarrow V$ such that
\begin{equation}\label{eta, eps}
\varphi(V_m)=\varepsilon^{-1}(V_m).
\end{equation} 
\end{definition} 
\noindent

If $\phi$ is a nonconstant standard extension, the subspace 
$U\subset V'$ is unique and the linear operator $\varepsilon:U
\to V$ is unique up to a scalar multiple. Indeed, assume 
$\phi$ is given and set 
\begin{equation}\label{descriptn  of W}
W:=\underset{V_m\subset V}{\bigcap} \varphi(V_m).
\end{equation} 
Let $S$ and $S'$ denote respectively
the tautological bundles on $G(m,V)$ and $G(n,V')$. There is an
obvious exact sequence
$$
0\to W\otimes\mathcal{O}_{G(m,V)}\to\phi^*S'\to S\to0.
$$
Dualization yields an injective homomorphism $V^{\vee}=
H^0(G(m,V),S^{\vee})\hookrightarrow H^0(G(m,V),(\phi^*S')
^{\vee})$ with cokernel equal $W^{\vee}$. Set $U^{\vee}=
H^0(G(m,V),(\phi^*S')^{\vee})$. Then a second dualization 
yields a surjective homomorphism $\varepsilon:U\to V$ with 
$\ker\varepsilon=W$. In particular,
\begin{equation}\label{descriptn  of U}
U=\underset{V_m\subset V}{\bigcup} \varphi(V_m).
\end{equation}

In what follows, we will assign a subspace $U\subset V'$ also 
in the case when $\phi$ is constant: we set $U=W:=\phi(G(m,V))
\in G(n,V')$ and $\varepsilon=0$. Formulas \eqref{eta, eps} and
\eqref{descriptn  of U} then hold in this case too. 

It is easy to show that Definition 
\ref{strict} is equivalent to the above "naive" definition of 
strict standard extension. Let $\phi$ be a nonconstant strict 
standard extension according to Definition \ref{strict}.  
Then $U$ and $\varepsilon:U\to V$ are given, and we can choose 
a splitting $U\simeq V\oplus(W=\ker
\varepsilon)$. In particular, this induces an embedding $V$ 
into $V'$. We then extend the splitting $U\simeq V\oplus W$ to
a splitting $V'=V\oplus\widehat{W}$ where $W\subset
\widehat{W}$. This yields the datum of "naive" definition.
Conversely, given a nonconstant strict standard extension as
in the "naive" definition, we simply set $U:=V\oplus W$ and 
define $\varepsilon$ to be the projection $U\to V$. Finally, 
if $\phi$ is constant then we put $U:=\phi(G(m,V))=W$ (here $\dim U=n$). 
%In this case, in the naive definition we have to set $V'=
%\widehat{W}$ (informally setting $V=0$ in the decomposition 
%$V'=V\oplus \widehat{W}$).

In the orthogonal and symplectic cases, in Definition 
\ref{strict} one must assume that the space $W$ is
isotropic and the isomorphism $U/W\xrightarrow{\sim}V$ induced 
by the operator $\varepsilon:U\twoheadrightarrow V$ is an
isomorphism of spaces endowed with symmetric, or respectively 
symplectic, forms. Here the form on $U$ is induced by the 
respective form on $V'$. 

It is a straightforward observation that in all cases the 
composition of standard extensions of grassmannians is also a
standard extension. The composition of two strict standard 
extensions or two modified standard extensions is a strict
standard extension, while the composition of a strict standard 
extension and a modified standard extension is again a 
modified standard extension.

We now give the definition of a strict standard extension of 
usual and isotropic flag varieties.

\begin{definition}\label{str st ext flags}
An embedding of flag varieties 
$\varphi:\ Fl(m_1,...,m_k,V)
\hookrightarrow Fl(n_1,...,n_{\tilde{k}},V')$, 
respectively, $\varphi:\ FlO(m_1,...,m_k,V)
\hookrightarrow FlO(n_1,...,n_{\tilde{k}},V')$, 
respectively, $\varphi:\ FlS(m_1,...,m_k,V)
\hookrightarrow FlS(n_1,...,n_{\tilde{k}},V')$,
is said to be a \textit{strict standard extension}, or simply 
a \textit{standard extension} in the orthogonal and symplectic 
cases, if there exists a flag of distinct nonzero subspaces of 
$V'$,
$$
U_1\subset U_2\subset...\subset U_{\tilde{k}}
$$
such that in the orthogonal and symplectic cases the spaces 
$U_i$ are nondegenerate, and a commutative diagram
\begin{equation}\label{Ui,epsilon i}
\xymatrix{
V \ar@{=}[r]& V \ar@{=}[r]& .\ .\ . \ar@{=}[r]& V\\
U_1\ar[u]^-{\varepsilon_1} \ar@{^{(}->}[r] & 
U_2\ar[u]^-{\varepsilon_2} \ar@{^{(}->}[r] & .\ .\ .  \ar@{^{(}->}[r]  & 
U_{\tilde{k}} \ar[u]^-{\varepsilon_{\tilde{k}}} } 
\end{equation}
of linear operators $\varepsilon_i:U_i\to V$,
surjective whenever nonzero, compatible with the respective
forms on $U_i$ and $V$ and having isotropic kernels in the
orthogonal and symplectic cases, and such that
\begin{equation}\label{phi(...)}
\begin{split}
& \varphi\big(0=V_{\overline{p}(0)}\subset V_{\overline{p}(1)}\subset...\subset V_{\overline{p}(\tilde{k})
}\subset V_{\overline{p}(\tilde{k}+1)}=V\big)=\\
& \big(0\subset\varepsilon_1^{-1}(V_{\overline{p}(1)})\subset
\varepsilon_1^{-1}(V_{\overline{p}(2)})\subset...\subset
\varepsilon_{\tilde{k}}^{-1}(V_{\overline{p}(\tilde{k})})
\subset V'\big)
\end{split}
\end{equation} 
for a suitable surjective map $\overline{p}:\{0,1,...,\tilde{k}
,\tilde{k}+1\}\to\{0,1,...,k,k+1\}$ satisfying $\overline{p}(i)
\le \overline{p}(j)$ whenever $i<j$. 

Note that $\overline{p}(0)=0,\ \overline{p}(\tilde{k}+1)=k+1$ 
and that there are exactly $k$ distinct proper nonzero 
subspaces among $V_{\overline{p}(1)},...,V_{\overline{p}
(\tilde{k})}$. Moreover, the surjection $p:\{0,1,...,\tilde{k}
\}\to\{0,1,...,k\}$ satisfies $p(j)=\overline{p}(j)$ whenever 
$p(j)\ne0$ and $p^{-1}(0)\cup\{\tilde{k}\}=\overline{p}^{-1}
(0)\sqcup\overline{p}^{-1}(k+1)$.

A strict standard extension is a linear embedding, except  in the case
$$
FlO(m_1,...,m_k,V)\hookrightarrow FlO(n_1,...,n_{\tilde{k}},V')
$$
where $\frac{\dim V}{2}-1$ appeas among $m_1,...,m_k$ but 
$\frac{\dim V'}{2}-1$ does not appear among $n_1,...,
n_{\tilde{k}}$, or $\frac{\dim V}{2}$ appears among $m_1,...,
m_k$ but $\frac{\dim V'}{2}-1$ or $\frac{\dim V'}{2}$ does not 
appear among $n_1,...,n_{\tilde{k}}$
. 

\end{definition}

Of course, in the case of ordinary (i.e. not isotropic) flag
varieties, we also need the definition of a \textit{modified 
standard extension}. By definition, this is a composition
$\phi=d\circ\phi'$ where
\begin{equation*}\label{phi modified}
\phi':\ Fl(m_1,...,m_k,V)\hookrightarrow Fl(\dim V'-
n_{\tilde{k}},..., \dim V'-n_1,V'^{\vee})
\end{equation*}
is a strict standard extension and
\begin{equation*}\label{duality}
d:\ Fl(\dim V'-n_{\tilde{k}},...,\dim V'-n_1,V'^{\vee})
\xrightarrow{\simeq} Fl(n_1,...,n_{\tilde{k}},V')
\end{equation*}
is the duality isomorphism. Here $\phi^*[M_j]=[L_{q(j)}]$ for
a map $q:\{0,1,...,\tilde{k}\}\to\{0,1,...,k\}$ such that $q(0)
=0$, $q(i)\ge q(j)$ whenever $q(i)\ne0,\ q(j)\ne0$ and $i\le 
j$, and also $q(j)=0$ implies $j<t$ or $j>t$ for all $t$ with 
$q(t)\ne0$.

\begin{example}\label{example 3.3}
(i) Consider the extreme case when $k=1$ and $\tilde{k}$ is an 
arbitrary integer greater or equal to 1. Then the surjection 
$\overline{p}:\{0,1,...,\tilde{k},\tilde{k}+1\}\to\{0,1,2\}$ 
from Definition \ref{str st ext flags},(ii)
defines an ordered partition of $\{0,1,...,\tilde{k},
\tilde{k}+1\}$ with three parts $\overline{p}^{-1}(0)$, 
$\overline{p}^{-1}(1)$, $\overline{p}^{-1}(2)$, and a 
corresponding standard extension 
$$
G(m,V)\hookrightarrow Fl(m_1,...,m_{\tilde{k}},V')
$$
has the form
$$
(0\subset V_m\subset V)\mapsto(0\subset W_1\subset...\subset 
W_s\subset\varepsilon_{s+1}^{-1}(V_m)\subset...\subset
\varepsilon_t^{-1}(V_m)\subset U_{t+1}\subset...\subset 
U_{\tilde{k}}\subset V'),
$$
where $\{0,1,...,s\}=\overline{p}^{-1}(0)$, $\{s+1,...,t\}=
\overline{p}^{-1}(1)$ and $\{t+1,...,\tilde{k}+1\}=
\overline{p}^{-1}(2)$. 
 
(ii) Next, consider the case when $\dim V'=\dim V+1$. Then
$\tilde{k}$ necessarily equals $k$ or $k+1$. Hence, $\dim 
W_i\le1$ and there exists $i_0,\ 0\le i_0\le k,$ such that 
$W_j=0$ for $j\le i_0$ and $\dim W_{i_0+1}=...=\dim 
W_{\tilde{k}}=1$. Consequently, $W_{i_0+1}=...=
W_{\tilde{k}}$. Set $W:=W_{i_0+1}=...=W_{\tilde{k}}$. If 
$\tilde{k}=k$, then $p$ is a bijection and the corresponding 
standard extension $\phi:\ Fl(m_1,...,m_k,V)\hookrightarrow 
Fl(n_1,...,n_k,V')$ has the form
\begin{equation}\label{example of st ext}
\begin{split}
&\ \ \ \ \ \varphi(0\subset V_{m_1}\subset...\subset 
V_{m_k}\subset V)=\\
& =\left\{ 
\begin{aligned}
(0\subset V_{m_1}\oplus W\subset...\subset V_{m_k}\oplus W\subset V')\ \ \ \ \ \ \ \ \ \ \ \ \ \ \ \ \ \ \ \ \ \ \ \ \ \
\ \ \ & &\ for\ i_0=0,\\ 
(0\subset V_{m_1}\subset...\subset V_{m_{i_0}}\subset V_{m
_{i_0+1}}\oplus W\subset...\subset V_{m_k}\oplus W\subset V')
 & &\ for\ 0<i_0<k,\\ 
(0\subset V_{m_1}\subset...\subset V_{m_k}\subset V')
\ \ \ \ \ \ \ \ \ \ \ \ \ \ \ \ \ \ \ \ \ \ \ \ \ \ \ \ \ \ \ \ \ \ \ \ \ \ \ \ \ \  & &
\ for\ i_0=k.\\ 
\end{aligned} 
\right.
\end{split}
\end{equation}
If $\tilde{k}=k+1$, then $p(i_0)=p(i_0+1)=i_0$ and the 
standard extension $\varphi:\ Fl(m_1,...,m_k,V)
\hookrightarrow Fl(n_1,...,n_{k+1},V')$ has the form
\begin{equation}\label{example of st ext2}
\begin{split}
&\ \ \ \ \ \varphi(0\subset V_{m_1}\subset...\subset 
V_{m_k}\subset V)=\\
& =\left\{ 
\begin{aligned}
(0\subset W\subset V_{m_1}\oplus W\subset...\subset V_{m_k}\oplus W\subset V')\ \ \ \ \ \ \ \ \ \ \ \ \ \ \ \ \ \ \ \ \ \ \ \ \  & &\ for\ i_0=0,\\ 
(0\subset V_{m_1}\subset...\subset V_{m_{i_0}}\subset V_{m
_{i_0+1}}\oplus W\subset...\subset V_{m_k}\oplus W\subset V')
\ \ \ \  & &\ for\ 0<i_0<k,\\ 
(0\subset V_{m_1}\subset...\subset V_{m_k} \subset V_{m_k}\oplus W\subset V')
\ \ \ \ \ \ \ \ \ \ \ \ \ \ \ \ \ \ \ \ \ \ \ \ \ \ \ \ \ \  & &
\ for\ i_0=k.\\ 
\end{aligned} 
\right.
\end{split}
\end{equation}

(iii) Let $\dim V=2$ and let $V'=V\oplus V$. Consider the 
embedding
$$
\mathbb{P}(V)=G(1,V)\hookrightarrow Fl(1,2,3,V\oplus V),\ \ \ 
\ \ \ (0\subset V_1\subset V)\mapsto(0\subset V_1\subset 
V\oplus0\subset V\oplus V_1\subset V\oplus V).
$$
This embedding is not a standard extension. Here, $\varphi^*[
M_1]=\varphi^*[M_3]=[L],\ \varphi^*[M_2]=0.$ This shows that 
there is no $p$ as in the definition of strict standard 
extension, and it is easy to check that $\phi$ is also not a 
modified standard extension.
\end{example}

(iv) Let $V'$ be endowed with non-degenerate symmetric or
symplectic form, and $V'=V\oplus\widehat{W}$ where $\widehat{W}
=V^{\bot}$ and $\dim\widehat{W}=2$. Fix an isotropic line $W
\subset\widehat{W}$. Then for any increasing sequence $0<m_1<
...<m_k\le[\frac{\dim V}{2}]$ and any $s$, $1\le s\le k$, 
there is a standard extension $\phi:X\to Y$, where $X=FlO(m_1,
...,m_k,V)$ and $Y=FlO(m_1,...,m_s,m_s+1,...,m_k+1,V')$, or 
respectively, $X=FlS(m_1,...,m_k,V)$ and $Y=FlS(m_1,...,m_s,
m_s+1,...,m_k+1,V')$. For $s=0$ there also is a standard 
extension $\phi:X\to Y$, where now $Y=FlO(1,m_1+1,...,m_k+1,V'
)$ or $Y=FlS(1,m_1+1,...,m_k+1,V')$, respectively. The 
embedding $\phi$ is given by formula \eqref{example of st ext2}
with $i_0$ substituted by $s$.
%$$
%\phi(0\subset V_{m_1}\subset...\subset V_{m_s}\subset 
%V_{m_s+1}\subset...\subset V_{m_k}\subset V)=
%$$
%$$
%(0\subset V_{m_1}\subset...\subset V_{m_s}\subset 
%V_{m_s}\oplus W\subset V_{m_s+1}\oplus W\subset...\subset V_{m_k}\oplus W\subset V').
%$$

\vspace{3mm}
A less canonical, but more intuitive, description of strict 
standard extensions (respectively, of standard extensions in 
the isotropic case) is given by the following easily proved 
proposition.
\begin{proposition}\label{descrn of st ext}
Assume that $\varphi:\ Fl(m_1,...,m_k,V)\hookrightarrow 
Fl(n_1,...,n_{\tilde{k}},V')$, respectively, $\varphi:\ 
FlO(m_1,...,m_k,V)\hookrightarrow FlO(n_1,...,n_{\tilde{k}},
V')$, respectively, $\varphi:\ FlS(m_1,...,m_k,V)
\hookrightarrow FlS(n_1,...,n_{\tilde{k}},V')$ is a 
nonconstant strict standard extension corresponding to a 
surjection $\overline{p}:\{0,1,...,\tilde{k},\tilde{k}+1\}\to
\{0,1,...,k,k+1\}$. Define the flag $(0\subset W_1\subset...
\subset W_{\tilde{k}}\subset V')$ by setting $W_i:=\ker
\varepsilon_i$. Then there exists a direct sum decomposition 
\begin{equation}\label{direct sum with W}
V'=V\oplus \widehat{W}
\end{equation}
with $\widehat{W}=V^{\bot}$ in the orthogonal and symplectic 
case, and such that $W_i\subset \widehat{W}$, $U_i\supset V$ 
for all $i$ with $\varepsilon_i\ne0$, and the nonzero 
operators $\varepsilon_i:U_i\to V$ are just projections onto 
$V$ via the decomposition \eqref{direct sum with W}. Moreover,
\begin{equation}\label{phi(...)1}
\varphi\big(0\subset V_{\overline{p}(1)}\subset...\subset 
V_{\overline{p}(\tilde{k})}\subset V\big)=
\big(0\subset V_{\overline{p}(1)}\oplus W_1\subset...\subset 
V_{\overline{p}(\tilde{k})}\oplus W_{\tilde{k}}\subset V'\big).
\end{equation} 
\end{proposition}

\begin{lemma}\label{Lemma 3.7}
In the notation of Proposition \ref{descrn of st ext}, let 
$\underline{w}$ be a basis of $\widehat{W}$ such that all 
subspaces $W_i$ are coordinate subspaces with respect to 
$\underline{w}$. Then, for any splitting $\widehat{W}=
\overline{W}\oplus\overline{\overline{W}}$ such that 
$\overline{W}$ and $\overline{\overline{W}}$ are
coordinate spaces, mutually perpendicular within $\widehat{W}$
in the orthogonal and symplectic cases, the strict standard 
extension given by formula \eqref{example of st ext2} is the 
composition of strict standard extensions
$$
Fl(m_1,...,m_k,V)\hookrightarrow Fl(m'_1,...,m'_l,V\oplus
\overline{W})\hookrightarrow Fl(n_1,...,n_{\tilde{k}},V'=
(V\oplus\overline{W})\oplus\overline{\overline{W}})
$$
for which the corresponding flags in $\overline{W}$ and 
$\overline{\overline{W}}$ are the respective intersections of 
the flag $(0\subset W_1\subset...\subset W_{\tilde{k}}
\subset W)$ with $\overline{W}$ and $\overline{\overline{W}}$.
\end{lemma}
\begin{proof}
Direct verification  using formula \eqref{phi(...)1}.
\end{proof}

\vspace{1cm}
\section{A sufficient condition for a linear embedding to be a 
standard extension}\label{more linear embed}
\vspace{5mm}

In this section we establish our main result concerning linear 
embeddings of flag varieties. This is a sufficient condition 
for a linear embedding to be a standard extension.

Consider a flag variety $Fl(m_1,...,m_k,V)$ and let $\{m_1,
...,m_k\}=R_1\cup...\cup R_s$ be a decomposition into a union 
of $s$ subsets. Denote this decomposition by $R$. By ordering 
the elements of $R_i$ we can think of $R_i$ as a type of a 
flag, and then $Fl(R_i,V)$ is a well-defined flag variety. 
Moreover, there is a canonical embedding 
$$
\psi_{R,t_1,...,t_s}:\ Fl(m_1,...,m_k,V)\hookrightarrow
Fl(R_1,V)^{\times t_1}\times...\times Fl(R_s,V)^{\times t_s}
$$
where by $Fl(R_i,V)^{t_i}$ we denote the direct product of 
$t_i$ copies of $Fl(R_i,V)$.

If now $\phi: Fl(m_1,...,m_k,V)\hookrightarrow Fl(n_1,...,
n_{\tilde{k}},V')$ is an embedding, we say that $\phi$ 
\textit{does not factor through any direct product} if $\phi\ne
\psi\circ\psi_{R,t_1,...,t_s}$ for any decomposition $R$, any
$t_i\in\mathbb{Z}_{\ge1}$ and any embedding 
$\psi:\ Fl(R_1,V)^{\times t_1}\times...\times Fl(R_s,V)^{
\times t_s}\hookrightarrow Fl(n_1,...,n_{\tilde{k}},V')$.
The definition clearly makes sense also in the orthogonal and
symplectic cases. 

\begin{lemma}\label{not extend}
Let $\phi:Fl(m_1,...,m_k,V)\hookrightarrow Fl(n_1,...,
n_{\tilde{k}},V')$ be a linear embedding which does not 
factor through any direct product. Assume that $\tilde{k}\ge
3$ and there exist integers $i$ and $j$, $1\le i,\ i+2\le 
j\le\tilde{k}$, such that the morphisms $\pi_i\circ\phi$ and 
$\pi_j\circ\phi$ are not constant maps. Then for any $l,\ 
i<l<j,$ the morphism $\pi_l\circ\phi$ is not a constant map. 
Similar statements are true in the orthogonal and symplectic 
cases.
\end{lemma}
\begin{proof}
Suppose the contrary, i. e. that there exists $l,\ i<l<j,$ 
such that the morphism $\pi_l\circ\phi$ is a constant map, and 
let $V'_l:=\mathrm{im}(\pi_l\circ\phi)\subset V'$. Then $\phi$ 
induces well-defined embeddings
$$
\phi': Fl(p(\{0,1,...,l\}),V)\hookrightarrow 
Fl(n_1,...,n_{\tilde{k}},V'),
$$
$$
\phi'': Fl(p(\{l,...,\tilde{k}\}),V)\hookrightarrow 
Fl(n_1,...,n_{\tilde{k}},V'),
$$
where we consider $p(\{0,1,...,l\})$ and $p(\{l,...,\tilde{k}\}
)$ as types of flags. Moreover, $\phi$ clearly factors through 
the embedding
$$
\psi:\ Fl(p(\{0,1,...,l\}),V)\times Fl(p(\{l,...,\tilde{k}\},V)
\to Fl(n_1,...,n_{\tilde{k}},V'),
$$
where, for $F_1\in Fl(p(\{0,1,...,l\}),V)$ and $F_2\in Fl(p(
\{l,...,\tilde{k}\},V)$, the spaces with indices from 1 to 
$l$ of the flag $\psi(F_1\times F_2)$ coincide with those of 
the flag $\phi'(F_1)$, and the spaces with indices from $l$ 
to $\tilde{k}$ coincide with those of the flag $\phi''(F_2)$. 
The flag $\psi(F_1,F_2)$ is well defined as its space with 
index $l$ equals $V'_l$.
\end{proof}

\begin{theorem}\label{Thm 4.6}
 Let $\phi:Fl(m_1,...,m_k,V)\hookrightarrow Fl(n_1,...,
 n_{\tilde{k}},V')$ be a linear embedding.
Assume that all morphisms $\varphi_{p(j),j}:\ G(m_{p(j)},V)
\hookrightarrow G(n_j,V')$ from Proposition \ref{extension to 
Grassm growth} are strict standard extensions, and that $\phi
$ does not factor through any direct product. Then $\varphi$ 
is a strict standard extension. Analogous statements hold in 
the orthogonal and symplectic cases.
\end{theorem}
\begin{proof} 
Lemma \ref{not extend} implies that there are $s$ and $t$, $s<t
$, so that $p(j)=0$ holds precisely for $j\le s$ and for $j\ge 
t$. 

In the case when there is a single index $j$ such 
that $\phi_{p(j),j}$ is a nonconstant morphism, the 
statement of the theorem is easy. We thus may assume that there 
are (at least) two indices $j$ and $j+1$, $1<j<j+1<t$, so that 
$\phi$ induces nonconstant strict standard extensions
\begin{equation*}\label{phi j,j+1}
\phi_{p(j),j}:\ G(m_{p(j)},V)\hookrightarrow G(n_j,V'),\ \ \ \ 
\ \ \phi_{p(j+1),j+1}:\ G(m_{p(j+1)},V)\hookrightarrow 
G(n_{j+1},V').
\end{equation*}

Define subspaces $U_j$ and $U_{j+1}$ of $V'$ by formula 
\eqref{descriptn of U} in which we put $\phi=\phi_{p(j),j}$
and $m=m_{p(j)}$, or $\phi=\phi_{p(j+1),j+1}$ and $m=m_{p(j+1)}
$, respectively. Let $(0\subset V_{m_1}\subset...\subset V_{m_k
}\subset V)$ denote an arbitrary point of $Fl(m_1,...,m_k,V)$. 
Since by definition
\begin{equation}\label{phi j in phi j+1}
\phi_{p(j),j}(V_{m_{p(j)}})\subset
\phi_{p(j+1),j+1}(V_{m_{p(j+1)}}) 
\end{equation}
for any subflag $V_{m_{p(j)}}\subset V_{m_{p(j+1)}}$ if $p(j)
<p(j+1)$, or for any subflag $V_{m_{p(j+1)}}\subset 
V_{m_{p(j)}}$ if $p(j+1)>p(j)$, formula \eqref{descriptn  
of U} implies that $U_j$ is a subspace of $U_{j+1}$. Next, 
since the strict standard extensions $\phi_{p(j),j}$ and 
$\phi_{p(j+1), j+1}$ are nonconstant, it follows from 
Definition \ref{strict} that there are surjective linear 
operators $\varepsilon_j:U_j\to V$ and $\varepsilon_{j+1}:
U_{j+1}\to V$, such that formula \eqref{eta, eps} holds for 
$\varepsilon=\varepsilon_j,\ m=m_{p(j)}$ and $\varepsilon=
\varepsilon_{j+1},\ m=m_{p(j+1)}$, respectively. This, 
together with \eqref{phi j in phi j+1}, means that
\begin{equation}\label{eps j in}
\varepsilon_j^{-1}(V_{m_{p(j)}})\subset
\varepsilon_{j+1}^{-1}(V_{m_{p(j+1)}}) 
\end{equation}
under the same conditions on $V_{m_{p(j)}}$ and $V_{m_{p(j+1)}
}$ as in \eqref{phi j in phi j+1}.

Denoting $W_j=\ker\varepsilon_j$ and $W_{j+1}=\ker\varepsilon
_{j+1}$, in view of \eqref{phi j in phi j+1} we obtain from 
\eqref{descriptn  of W} that $W_j$ is a subspace of $W_{j+1}$. 
The inclusions $U_j\subset U_{j+1}$ and $W_j\subset W_{j+1}$ join into a commutative diagram
\begin{equation}\label{2new Uij,Wij}
\xymatrix{
V\ar[r]^-{\theta_j} & V \\
U_j\ar@{^{(}->}[r]\ar[u]^-{\varepsilon_j} & 
U_{j+1}\ar[u]^-{\varepsilon_{j+1}}\\
W_j\ar@{^{(}->}[r]\ar@{^{(}->}[u] & W_{j+1},\ar@{^{(}->}[u]}
\end{equation}
where $\theta_j$ is the induced linear operator. From 
\eqref{eps j in} and \eqref{2new Uij,Wij} we obtain
\begin{equation}\label{theta j in}
\theta_j(V_{m_{p(j)}})\subset V_{m_{p(j+1)}}.
\end{equation}

Now we are going to show that 
\begin{equation*}\label{p(j)le p(j+1)}
p(j)\ \le\ p(j+1). 
\end{equation*}
Assume the contrary, i.e. $p(j+1)<p(j)$. Then the inclusion
\eqref{theta j in} implies
$$
\theta_j(V_{m_{p(j)}}) \subset \underset{V_{m_{p(j+1)}}\subset 
V_{m_{p(j)}}}{\bigcap} V_{m_{p(j+1)}}=0.
$$
Thus $\theta_j=0$, and consequently $U_j\subset W_{j+1}$ by
diagram \eqref{2new Uij,Wij}. This together with formula  
\eqref{eta, eps} means that the inclusion \eqref{eps j in}
extends to a pair of inclusions
\begin{equation*}\label{eps j in W}
\varepsilon_j^{-1}(V_{m_{p(j)}})\subset W_{j+1}\subset
\varepsilon_{j+1}^{-1}(V_{m_{p(j+1)}}),\ \ \ \ 
\end{equation*}
for any $(V_{m_{p(j)}},V_{m_{p(j+1)}})\in G(m_{p(j)},V)\times 
G(m_{p(j+1)},V)$. Then the exact same argument as in the 
proof of Lemma \ref{not extend} shows that $\phi$ factors 
through a direct product. Hence the assumption $p(j+1)<p(j)$ 
is invalid.

Next, we claim that $\theta_j=c_j\mathrm{Id}$ for some nonzero 
constant $c_j$. Note that $\theta_j\ne0$ by the above. Then, 
since $\varepsilon_j^{-1}(V_{m_{p(j)}})\subset\varepsilon_{j+1}
^{-1}(V_{m_{p(j+1)}})$,
%any inclusion $V_{m_{p(j)}}\subset V_{m_{p(j+1)}}$ of 
%subspacesof $V$ of respective dimensions $m_{p(j)}$ and 
%$m_{p(j+1)}$,
we have $\theta_j(V_{m_{p(j)}})\subset V_{m_{p(j+1)}}$. Taking
into account that $\underset{V_{m_{p(j+1)}}\supset 
V_{m_{p(j)}}}{\bigcap}V_{m_{p(j+1)}}=V_{m_{p(j)}}$, we obtain
\begin{equation*}\label{theta(...)in}
\theta_j(V_{m_{p(j)}})\subset V_{m_{p(j)}}
\end{equation*}
for any $V_{m_{p(j)}}\in G(m_{p(j)},V)$. As any 1-dimensional 
subspace of $V$ is the intersection of all 
$m_{p(j)}$-dimensional subspaces which contain it, we see that
any vector in $V$ is an eigenvector for $\theta_j$. 
Consequently, we have $\theta_j=c_j\mathrm{Id}$ for $c_j\ne0$.

The above argument applies to any pair of integers $j,j+1$ 
where $s+1<j<t-2$. Therefore, we can construct a commutative 
diagram
\begin{equation}\label{Ui,epsilon i new}
\xymatrix{
V \ar[r]^-{\theta_1} & V \ar[r] & .\ .\ .\ar[r] &
V \ar[r]^-{\theta_{\tilde{k}}}  & V\\
U_1\ar[u]^-{\varepsilon_1}\ar@{^{(}->}[r] & 
U_2\ar[u]^-{\varepsilon_2}\ar@{^{(}->}[r] & .\ .\ .
\ar@{^{(}->}[r] & U_{\tilde{k}-1}\ar[u]^-{\varepsilon_{\tilde{k}-1}}\ar@{^{(}->}
[r] & U_{\tilde{k}} \ar[u]^-{\varepsilon_{\tilde{k}}} ,} 
\end{equation}
where the morphisms $\varepsilon_i$ equal zero for $i\le s$,
$i\ge t$, $\theta_i=\mathrm{Id}$ for $i\le s$ and $i\ge t$, 
and $\theta_i=c_i\mathrm{Id}$ with $c_i\ne0$ for $s+1\le i\le 
t-1$. Here, the spaces $U_1,...,U_s,U_{t+1},...,U_{\tilde{k}}$
are defined as the subspaces of $V'$ which equal the images of
the respective constant morphisms $\pi'_1\circ\phi,..., \pi'_s
\circ\phi$, $\pi'_{t+1}\circ\phi,...,\pi'_{\tilde{k}}
\circ\phi$, where 
$$
\pi'_r: \ Fl(n_1,...,n_{\tilde{k}},V')\to G(n_r,V')
$$
are the natural projections.

Via scaling the morphisms $\varepsilon_i$ for $s+1\le i
\le t-1$, we can turn the diagram \eqref{Ui,epsilon i new} 
into the diagram \eqref{Ui,epsilon i} in the definition of 
strict
standard extension. An immediate checking shows that our given
embedding $\phi$ is given by formula \eqref{phi(...)} for the
surjection $\overline{p}:\{0,1,...,\tilde{k},\tilde{k}+1\}\to
\{0,1,...,k,k+1\}$ where $\overline{p}(j)=p(j)$ for $j\le 
t-1$, $\overline{p}(j)=\tilde{k}+1$ for $j\ge t$.
\end{proof}

The next theorem is a more general version of Theorem \ref{Thm 
4.6}.

\begin{theorem}\label{Thm 4.3}
If, in the setting of Theorem \ref{Thm 4.6}, all morphisms 
$\phi_{p(j),j}$ are (not necessarily strict) standard 
extensions, then $\phi$ is also a standard extension.
\end{theorem}

\begin{proof}
First, as in the proof of Theorem \ref{Thm 4.6}, we assume 
that there are (at least) two indices $j$ 
and $j+1$ such that there are nonconstant standard extensions 
$\phi_{p(j),j}$ and $\phi_{p(j+1),j+1}$ as in \eqref{phi 
j,j+1}. The reader will easily handle the remaining case. 
%The case of a strict standard extension is covered by Theorem
%\ref{Thm 4.6}. 

We will show now that the standard extensions    
$\phi_{p(j),j}$ and $\phi_{p(j+1),j+1}$ are either both strict 
or are both modified. For this, we need to exclude the following other logical possibilities:\\
(a) $p(j)\le p(j+1)$, $\phi_{p(j),j}:\ G(m_{p(j)},V)
\hookrightarrow G(n_j,V')$ is a strict standard extension and 
$\phi_{p(j+1),j+1}:\ G(m_{p(j+1)},V)\hookrightarrow G(n_{j+1},
V')$ is a modified standard extension; \\
(b) $p(j)>p(j+1)$, $\phi_{p(j),j}$ is a modified standard 
extension and $\phi_{p(j+1),j+1}$ is a strict standard 
extension; \\ 
(c) $p(j)\le p(j+1)$, $\phi_{p(j),j}$ is a modified 
standard extension and $\phi_{p(j+1),j+1}$ is a strict 
standard extension; \\
(d) $p(j)>p(j+1)$, $\phi_{p(j),j}$ is a strict standard 
extension and $\phi_{p(j+1),j+1}$ is a modified standard 
extension. 

(a) Note that the modified standard extension $\phi_{p(j+1),
j+1}$ defines a flag of subspaces $W_{j+1}\subset U_{j+1}$ of 
$V'$ and a surjective linear operator
$\varepsilon_{j+1}:U_{j+1}\to V'^{\vee}$ with $\ker
\varepsilon_{j+1}=W_{j+1}$, such that 
\begin{equation}\label{eps modif}
\varphi_{p(j+1),j+1}(V_{m_{p(j+1)}})=\varepsilon_{j+1}^{-1}
((V/V_{m_{p(j+1)}})^{\vee}),
\end{equation} 
where $(V/V_{m_{p(j+1)}})^{\vee}$ is naturally considered as a 
subspace of $V^{\vee}$. Moreover,
\begin{equation}\label{descriptn  of U,W modif}
W_{j+1}=\underset{V_{m_{p(j+1)}}\subset V}{\bigcap} 
\varphi_{p(j+1),j+1}(V_{m_{p(j+1)}}).
\end{equation}
Formulas \eqref{eps modif} and \eqref{descriptn  of U,W modif} are corollaries of formulas \eqref{eta, eps} and 
\eqref{descriptn  of W}, respectively.

Now, given $V_{m_{p(j)}}\in G(m_{p(j)},V)$, we obtain
\begin{equation}\label{0=}
\{0\}=\underset{V_{m_{p(j+1)}}\supset V_{m_{p(j)}}}{\bigcap}
(V/V_{m_{p(j+1)}})^{\vee}, 
\end{equation}
where the intersection is taken in $(V/V_{m_{p(j)}})
^{\vee}$. Using \eqref{eps modif}-\eqref{0=}, we 
find $W_{j+1}=\underset{V_{m_{p(j+1)}}\supset V_{m_{p(j)}}}
{\bigcap}\phi_{p(j+1),j+1}(V_{m_{p(j+1)}})$. 
Therefore, 
\begin{equation}\label{phi subset}
\phi_{p(j),j}(V_{m_{p(j)}})\subset W_{j+1}\subset 
\varphi_{p(j+1),j+1}(V_{m_{p(j+1)}})
\end{equation}
for any $V_{m_{p(j+1)}}\in G(m_{p(j+1)},V)$. In view of 
\eqref{eta, eps} and \eqref{eps modif}, the inclusion
\eqref{phi subset} coincides with the inclusion \eqref{eps j 
in W}. Hence, as in the proof of Theorem \ref{Thm 4.6}, we see 
that $\phi$ factors through a direct product, contrary to our 
assumption. This contradiction rules out (a).

(b) Given $V_{m_{p(j+1)}}\in G(m_{p(j+1)},V)$, for any 
$V_{m_{p(j)}}\subset V_{m_{p(j+1)}}$ we have 
$\phi_{p(j+1),j+1}(V_{m_{p(j)}})\supset\phi_{p(j),j}
(V_{m_{p(j+1)}})$. Hence, there is an inclusion $\phi_{p(j),j}
(V_{m_{p(j+1)}})\subset\underset{V_{m_{p(j)}}\subset 
V_{m_{p(j+1)}}}{\bigcap}\phi_{p(j+1),j+1}(V_{m_{p(j)}})$, the 
right-hand side of which is zero, as it clearly follows from
the definition of nonconstant strict standard extension. Thus,
$\phi_{p(j),j}(V_{m_{p(j+1)}})=\{0\}$ which is a contradiction,
since $V_{n_1}\ne0$.

Cases (c) and (d) are reduced to cases (a) and  (b), 
respectively, via the duality isomorphisms $G(n_j,V')
\xrightarrow{\simeq} G(\dim V'-n_j,V'^{\vee})$ and $G(n_{j+1},
V')\xrightarrow{\simeq} G(\dim V'-n_{j+1},V'^{\vee})$. 
Thus, all the cases (a)-(d) lead to a contradiction. 

The above,
together with Lemma \ref{not extend}, implies that either all
nonconstant morphisms $\phi_{p(j),j}:\ G(m_{p(j)},V)
\hookrightarrow G(n_j,V')$ are strict standard extensions, or
that they all are modified standard extensions. In the latter 
case one considers the morphism  $d\circ\phi$, where $d$ is the
duality isomorphism. Then by Theorem \ref{Thm 4.6}, 
$d\circ\phi$ is a strict standard extension, and consequently
$\phi$ is a modified standard extension.
\end{proof}

We now introduce the following condition on a linear embedding
$$
\phi:Fl(m_1,...,m_k,V)\hookrightarrow Fl(n_1,...,
n_{\tilde{k}},V),
$$ 
or respectively,
$$
\phi:FlO(m_1,...,m_k,V)\hookrightarrow FlO(n_1,...,
n_{\tilde{k}},V)
$$
or
$$
\phi:FlS(m_1,...,m_k,V)\hookrightarrow FlS(n_1,...,
n_{\tilde{k}},V).
$$

(c) \textit{No nonconstant morphism $\phi_{p(j),j}:G(m_i,V)\to 
G(n_j,V')$ factors through an embedding of a projective 
subspace into $G(n_j,V')$; in the orthogonal and symplectic 
cases no nonconstant morphism $\phi_{p(j),j}:X\to Y$ for 
$X=GO(m_i,V)$ and $Y=GO(n_j,V')$, or $X=GS(m_i,V)$ and 
$Y=GS(n_j,V')$, factors through a smooth subvariety of $Y$ 
isomorphic to a grassmannian $G(m,V'')$ or a multidimensional 
quadric in case $Y=GO(n_j,V')$; in the case where $X=GO(s-1,V)
,\ Y=GO(t-1,V')$ for $\dim V=2s,\ \dim V'=2t$ for $t>s$, this 
latter condition should also be imposed on the induced 
morphism $\tilde{\phi}_{p(j),j}:GO(s,V)\to GO(t,V')$.}

We say that a linear embedding $\phi$ is \textit{admissible} if
it does not factor through any direct product and satisfies
condition (c).

Our main result in this section is the following.

\begin{corollary}\label{Cor 4.4}
An admissible linear embedding $\phi$ is a standard extension.
\end{corollary}

\begin{proof}
According to Theorem \ref{Thm 4.3}, all we need to show is that
condition (c) implies that every nonconstant morphism 
$\phi_{p(j),j}$ is a standard extension. For usual 
grassmannians, this follows directly from \cite[Thm. 1]{PT}, 
which claims that a linear morphism of grassmannians $\phi_{p(j
),j}:X\to Y$ is a standard extension unless it factors through a projective subspace of $Y$. For isotropic
grassmannians, \cite[Thm. 1]{PT} applies only to the case when
$\mathrm{Pic}X\simeq\mathrm{Pic}Y\simeq\mathbb{Z}$, and also
implies our claim under this assumption. It remains to
consider the situation of a linear morphism $\phi_{p(j),j}:
G(s-1,V)\to G(t-1,V')$ where $\dim V=2s,\ \dim V'=2t,\ t\ge 
s$. In this situation, as stated in Section 
\ref{preliminaries}, we always have a commutative diagram
\begin{equation*}\label{diag with theta}
\xymatrix{
GO(s-1,V)\ar@{^{(}->}[rrr]^-{\phi_{p(j),j}}\ar[d]^-{\theta} 
& & & GO(t-1,V')\ar[d]^-{\theta'} \\
GO(s,V) \ar@{^{(}->}[rrr]^-{\tilde{\phi}_{p(j),j}} & & & 
GO(t,V').}
\end{equation*}

Here, \cite[Thm. 1]{PT} applies to the linear morphism
$\tilde{\phi}:=\tilde{\phi}_{p(j),j}$, implying that it is a 
standard extension whenever it does not factor through a 
grassmannian or a multidimensional quadric embedded in 
$GO(t,V')$. Let this standard extension have the form 
\begin{equation}\label{lower st ext}
V_s\mapsto V_s\oplus W',
\end{equation}
where $V'=V\oplus W$ is an orthogonal decomposition and $W'$ is
a maximal isotropic subspace of $W$. We will show that 
$\phi:=\phi_{p(j),j}$ is the standard extension 
\begin{equation}\label{upper st ext}
V_{s-1}\mapsto V_{s-1}\oplus W'.
\end{equation}
For this, consider an arbitrary projective line $\mathbb{P}^1$
on $GO(s,V)$, i.e. a smooth rational curve $C\subset GO(s,V)$ 
such that $\mathcal{O}_{GO(s,V)}(1)|_C\simeq\mathcal{O}_{
\mathbb{P}^1}(1)$. It is an exercise to see that there exists 
an isotropic subspace $W_{\mathbb{P}^1}\subset V$ of dimension 
$p-2$, such that the restriction $E:=\mathcal{S}|
_{\mathbb{P}^1}$ of the tautological bundle $\mathcal{S}$ on 
$GO(s,V)$ is isomorphic to $2\mathcal{O}_{\mathbb{P}^1}(-1)
\oplus W_{\mathbb{P}^1}\otimes\mathcal{O}_{\mathbb{P}^1}
$. Hence, by \eqref{lower st ext}, we have 
\begin{equation}\label{restr S'}
E':=\phi^*\mathcal{S}'|_{\mathbb{P}^1}\simeq
2\mathcal{O}_{\mathbb{P}^1}(-1)\oplus(W_{\mathbb{P}^1}\oplus W')\otimes\mathcal{O}_{\mathbb{P}^1},
\end{equation}
where $\mathcal{S}'$ is the tautological bundle on $GO(t-1,V')$. 

For any point $x\in\mathbb{P}^1$, consider the projective 
spaces $\theta^{-1}(x)=\mathbb{P}(E^{\vee}|_t)$ and 
$\theta'^{-1}(\tilde{\phi}(x))=\mathbb{P}((E')^{\vee}|_t)$. By 
definition, $\phi|_{\theta^{-1}(x)}:\theta^{-1}(x)\to
\theta'^{-1}(\tilde{\phi}(x))$ is a linear embedding of 
projective spaces, hence it has the form
\begin{equation}\label{new st extn}
V_{s-1}\mapsto V_{s-1}\oplus W''(x)
\end{equation}
for some unique isotropic vector subspace $W''(x)\subset V'$. 
Indeed, $W''(x)=\underset{V_{s-1}
\in\theta^{-1}(x)}{\bigcap} \varphi(V_{s-1})$ (see 
\eqref{descriptn  of W}). Moreover, by construction, 
$W'':=\{(x,W''(x))\}_{x\in\mathbb{P}^1}$ is a vector subbundle 
of $E'$, and the condition that $\phi^*\mathcal{O}_{GO(t-1,V')}
(1)\cong\mathcal{O}_{GO(s-1,V)}(1)$ (see Example 
\ref{lin emb isotr grass}) implies
\begin{equation}\label{det W''}
\det W''\cong\mathcal{O}_{\mathbb{P}^1}.
\end{equation}

Consider the composition of morphisms of sheaves:
$f:W''\stackrel{i}{\hookrightarrow}E'\stackrel{pr}{\to}
2\mathcal{O}_{\mathbb{P}^1}(-1)$ where $i$ is the above 
mentioned monomorphism and $pr$ is the canonical projection
defined by \eqref{restr S'}. If $f$ is a nonzero morphism, it 
follows from \eqref{det W''} and Grothendieck's Theorem that
$W''$ contains a direct summand $\mathcal{O}_{\mathbb{P}^1}
(a)$ for some $a>0$. But this contradicts to \eqref{restr S'}
since $i$ is a monomorphism. Hence, $f=0$, and by
\eqref{restr S'}, $W''$ is a subbundle of the trivial bundle
$(W_{\mathbb{P}^1}\oplus W')\otimes\mathcal{O}_{\mathbb{P}^1}$.
Therefore, in view of \eqref{det W''}, $W''$ is itself a 
trivial bundle. This means that the space $W''(x)$ does not 
depend on $x\in\mathbb{P}^1$, but possibly  depends only on 
the choice of projective line $\mathbb{P}^1$. We can set 
$W''(x)=W''_{\mathbb{P}^1}$. Then
\begin{equation}\label{W'' subset...}
W''_{\mathbb{P}^1}\subset W_{\mathbb{P}^1}\oplus W'.
\end{equation}

Pick a point $x_0\in\mathbb{P}^1$, so that $W''(x_0)=W''
_{\mathbb{P}^1}$. Next, pick another line $\mathbb{P}'^1$ 
through $x_0$, distinct from $\mathbb{P}^1$. Then 
$W''_{\mathbb{P}^1}=W''_{\mathbb{P}'^1}$. 
Since, as one easily checks, any two points in $GO(s,V)$ can be
connected by a chain of projective lines, we conclude that 
$W''_{\mathbb{P}^1}$ does not depend on the line 
$\mathbb{P}^1$. We therefore denote this space by $W''_0$, and 
the inclusion \eqref{W'' subset...} can be rewritten as
\begin{equation}\label{W'' in}
W''_0\subset W_{\mathbb{P}^1}\oplus W',\ \ \ \ \ \ \ 
\mathbb{P}^1\subset GO(s,V).
\end{equation} 

Now one easily observes that $\underset{\mathbb{P}^1\subset 
GO(s,V)}{\bigcap}W_{\mathbb{P}^1}=\{0\}$. Hence, \eqref{W'' in}
implies $W_0''=\underset{\mathbb{P}^1\subset GO(s,V)}{\bigcap}
(W_{\mathbb{P}^1}\oplus W')=W'$. It follows that the linear
embedding $\phi$ in \eqref{new st extn} is $V_{s-1}\mapsto 
V_{s-1}\oplus W'$, i.e., $\phi$ coincides with 
\eqref{upper st ext} as claimed.
\end{proof}

Corollary \ref{Cor 4.4} provides a sufficient condition, in 
terms of pure algebraic geometry, for a linear embedding of
flag varieties, or varieties of isotropic flags, to be a
standard extension.

\vspace{1cm}
\section{Admissible direct limits of linear embeddings of flag varieties are isomorphic to ind-varieties of generalized flags}
\label{ind-var}
\vspace{5mm}

We start by recalling the notions of generalized flag and  
ind-variety of generalized flags introduced in 
\cite[Section 5]{DP}. 
Let $V$ be an arbitrary vector space. A {\it chain of 
subspaces in $V$} is a set $\mathcal{C}$ of pairwise distinct 
subspaces of $V$ such that for any pair $F$, $H\in\mathcal{C}
$, one has either $F\subset H$ or $H\subset F$. Every chain of 
subspaces $\mathcal{C}$ is linearly ordered by inclusion. 
Given a chain $\mathcal{C}$, we denote by $\mathcal{C}'$ 
(respectively, by $\mathcal{C}''$) the subchain of 
$\mathcal{C}$ that consists of all subspaces $C \in 
\mathcal{C}$ which have an immediate successor (respectively, 
an immediate predecessor) with respect to this ordering. 

A {\it generalized flag in $V$} is a chain of subspaces 
$\mathcal{F}$ that satisfies the following conditions: \\
(i) each $F\in\mathcal{F}$ has an immediate successor or an 
immediate predecessor, i.e. $\mathcal{F}=\mathcal{F}'\cup
\mathcal{F}''$; \\
(ii) $V\backslash\{0\}=\cup_{F'\in\mathcal{F}'} F''\backslash 
F'$, where $F''\in\mathcal{F}''$ is the immediate successor of 
$F'\in \mathcal{F}'$.

In what follows, we assume that $V$ is a countable-dimensional 
vector space with basis $E=\{e_n\}_{n\in\mathbb{Z}_{>0}}$. A
generalized flag $\mathcal{F}$ in $V$ is \textit{compatible 
with the basis} $E$ if for every $F\in\mathcal{F}$ the set 
$F\cap E$ is a basis of $F$. We say that a generalized flag 
$\mathcal{F}$ is {\it weakly compatible with $E$}, if 
$\mathcal{F}$ is compatible with some basis $L$ of $V$ such 
that $E \backslash (E \cap L)$ is a finite set. 

\begin{example}
Let $V=\mathrm{Span}E$ where $E=\{e_n\}_{n\in
\mathbb{Z}_{>0}}$.

(i) Any finite chain $(0\subset F_1\subset...\subset F_k
\subset V)$ of coordinate subspaces (i. e. subspaces $F_i
\subset V$ satisfying $F_i=
\mathrm{Span}\{F_i\cap E\}$ for $1\le i\le k$) is a 
generalized flag compatible with the basis $E$. If $\dim F_i<
\infty$ for $1\le i\le k$, and if one drops the condition that 
all $F_i$ are coordinate subspaces, then the chain 
$(0\subset F_1\subset...\subset F_k\subset V)$ is a 
generalized flag weakly compatible with $E$. 

(ii) Fix a bijection $\mathbb{Z}_{>0}=\mathbb{Z}_{>0}\sqcup 
\mathbb{Z}_{<0}$, and let $\prec$ denote the linear order on 
$\mathbb{Z}_{>0}$, induced by the obvious linear order on 
$\mathbb{Z}_{>0}\sqcup\mathbb{Z}_{<0}$ in which all elements of
$\mathbb{Z}_{<0}$ are larger than all elements of 
$\mathbb{Z}_{>0}$. Then the chain $\{0,F_j,V\}_{j\in
\mathbb{Z}_{>0}}$, where $F_j=\{\mathrm{Span}\{e_i\}
_{i\preccurlyeq j}\}$, is a generalized flag compatible with 
$E$.

(iii) Fix a bijection $\mathbb{Z}_{>0}=\mathbb{Q}_l\sqcup 
\mathbb{Q}_r$, where $\mathbb{Q}_l=\mathbb{Q}=\mathbb{Q}_r$, and consider the following linear order on $\mathbb{Q}_l\sqcup 
\mathbb{Q}_r$: \ $j\prec t\ \Leftrightarrow\ j\in\mathbb{Q}_l
\sqcup\mathbb{Q}_r,\ t\in\mathbb{Q}_l\sqcup \mathbb{Q}_r,
\ j<t$, or $j=t,\ j\in\mathbb{Q}_l,\ t\in\mathbb{Q}_r$. Then
the chain $\{F'_j,F''_j\}_{j\in\mathbb{Q}_l}$, where 
$F'_j=\mathrm{Span}\{e_k\}_{k\prec j}$, $F''_j=\mathrm{Span}
\{e_k\}_{k\preccurlyeq j}$, is a
generalized flag compatible with $E$.
\end{example}

We define two generalized flags $\mathcal{F}$ and 
$\mathcal{G}$ in $V$ to be {\it $E$--commensurable} if both 
$\mathcal{F}$ and $\mathcal{G}$ are weakly compatible with $E$ 
and there exists an inclusion preserving bijection $\varphi:
\mathcal{F}\to\mathcal{G}$ and a finite-dimensional subspace 
$U \subset V$, such that for every $F \in \mathcal{F}$
\begin{equation*}\label{E-commens}
F\subset\varphi(F)+U,\ \ \ \varphi(F)\subset F+U,\ \ \ 
\dim(F\cap U)=\dim(\varphi(F)\cap U).
\end{equation*}

Let 
\begin{equation*}\label{X=F(...)}
\mathbf{X}=\mathbf{Fl}(\mathcal{F},E,V)
\end{equation*}
denote the set of all generalized flags in $V$ that are
$E$-commensurable 
with $\mathcal{F}$. We now explain that $\mathbf{X}$ 
has a natural ind-variety structure. Let $V'_n:=\mathrm{Span}
\{e_j|j\le n\}$. Then the intersection $\mathcal{F}\cap V'_n$ 
is a flag in $V'_n$, and let this flag have type $0<m'_{n,1}
<...<m'_{n,k_n}<n$ for $k_n\le n-1$. Since $\dim V'_{n+1}=\dim 
V'_n+1=n+1$, if we set $W'_n:=\mathrm{Span}\{e_{n+1}\}$,
we have $V'_{n+1}=V'_n\oplus W'_n$ and there is a standard
extension $i_n:Fl(m'_{n,1},...,m'_{n,k_n},V'_n)\hookrightarrow
Fl(n'_{n+1,1},...,n'_{n+1,k_{n+1}},V'_{n+1})$ given by 
formulas \eqref{example of st ext} or \eqref{example of st 
ext2} in Example 3.4 (where we had no need to use as many    
subscripts as well as primes).

Note that this standard extension $i_n$ is determined by the 
two types of flags $(m'_{n,1},...,m'_{n,k_n})$ and 
$(n'_{n+1,1},..,n'_{n+1,k_{n+1}})$, and by the choice of 
$W'_{n+1}$. In \cite{DP} it is shown that $\mathbf{Fl}(\mathcal{F},E,V)$ is 
naturally identified with the direct limit 
$$
\varinjlim Fl(m'_{n,1},...,m'_{n,k_n},V'_n)
$$ 
of the embeddings $i_n$. In particular, this equips 
$\mathbf{Fl}(\mathcal{F},E,V)$ with the structure of an 
ind-variety.

Let's now consider the case when $V$ is endowed a nondegenerate
symmetric or symplectic bilinear form $(\ ,\ )$. Here we assume
that either the basis $E$ is isotropic and is enumerated as 
$\{e_n,e^n
\}_{n\in\mathbb{Z}_{>0}}$ where $(e_n,e^n)=1$ for $n\in
\mathbb{Z}_{>0}$, or that $E$ is enumerated as $\{e_n,e_0,e^n
\}_{n\in
\mathbb{Z}_{>0}}$ where $e_n$ and $e^n$ are isotropic vectors 
satisfying $(e_n,e^n)=1$ for $n\in\mathbb{Z}_{>0}$ and $e_0$
satisfies $(e_0,e_n)=(e_0,e^n)=0$, $(e_0,e_0)=1$. This latter
enumeration of $E$ is possible only in the case of a symmetric 
form. We define a generalized flag $\mathcal{F}$ to be 
\textit{isotropic} if it consists of isotropic and coisotropic 
subspaces (a subspace $F$ is \textit{coisotropic} if 
$F^{\bot}$ is isotropic) and is invariant under taking 
orthogonal complement. In the current case, where $\dim 
V=\infty$, this definition is more convenient for our purposes
than the consideration of "purely isotropic" flags as in 
Sections \ref{preliminaries}, \ref{linear embed} and \ref{more 
linear embed}. Note that an isotropic generalized flag is
determined by its subchain of isotropic spaces.

\begin{example}\label{Example 5.2}
Consider the case where $V$ is endowed with a nondegenerate
symmetric form and the basis of $V$ is enumerated as $\{e_n,
e_0,e^n\}_{n\in\mathbb{Z}_{>0}}$ as above. Set $F^l_j=
\mathrm{Span}\{e_n\}_{n>j,j\ge0}$, $F^r_j=(F^l_j)^{\bot}$. 
Then $F^l_j\supset F^l_k,\ F^r_j\subset F^r_k,\ F^l_j\subset 
F^r_k$ for $k\ge j$, and $\{F^l_j,F^r_j\}_{j\in\mathbb{Z}
_{\ge0}}$ is  a maximal isotropic generalized flag compatible 
with $E$.
\end{example}

By $\mathbf{FlO}(\mathcal{F},E,V)$, or respectively 
$\mathbf{FlS}(\mathcal{F},E,V)$, we denote the set of all 
generalized flags which are $E$-commensurable with a fixed 
isotropic flag $\mathcal{F}$ compatible with $E$. To define an 
ind-variety structure on $\mathbf{FlO}
(\mathcal{F},E,V)$ or $\mathbf{FlS}(\mathcal{F},E,V)$, set
$V'_n=\mathrm{Span}\{e_j,e^j\}_{j\le n}$ or respectively
$V'_n=\mathrm{Span}\{e_j,e_0,e^j\}_{j\le n}$. Then 
$\mathcal{F}\cap V'_n$ has an isotropic subflag of type
$0<m'_{n,1}<...<m'_{n,k_n}\le [\frac{n}{2}]$, and there is a
standard extension
$$
\psi_n:\ FlO(m'_{n,1},...,m'_{n,k_n},V'_n)\hookrightarrow
FlO(m'_{n+1,1},...,m'_{n+1,k_{n+1}},V'_n)
$$
or
$$
\psi_n:\ FlS(m'_{n,1},...,m'_{n,k_n},V'_n)\hookrightarrow
FlS(m'_{n+1,1},...,m'_{n+1,k_{n+1}},V'_{n+1}),
$$
determined uniquely by the isotropic 1-dimensional subspace
$W_n=\mathrm{Span}\{e_{n+1}\}$. One can show that the direct
limit of the embeddings $\psi_n$ is identified with
$\mathbf{FlO}(\mathcal{F},E,V)$, or respectively $\mathbf{FlS}
(\mathcal{F},E,V)$, and hence $\mathbf{FlO}(\mathcal{F},E,V)$ 
and $\mathbf{FlS}(\mathcal{F},E,V)$ are ind-varieties 
\cite{DP}.

Next, we will relate an arbitrary direct limit of strict
standard extensions to the ind-varieties $\mathbf{Fl}
(\mathcal{F},E,V)$, $\mathbf{FlO}(\mathcal{F},E,V)$, or
$\mathbf{FlS}(\mathcal{F},E,V)$. First, consider a chain of 
strict standard extensions
\begin{equation}\label{phi_n's}
\phi_N:Fl(m_{N,1},..., m_{N,k_N},V_N)\hookrightarrow Fl(
m_{N+1,1},...,m_{N+1,k_{N+1}},V_{N+1}) 
\end{equation}
for some choice of vector spaces $V_N$, $\dim V_{N+1}>\dim V_N$
for $N\in\mathbb{Z}_{>0}$. Then, according to Proposition 
\ref{descrn of st ext}, we may choose vector spaces 
$\widehat{W}_N$, together with isomorphisms
\begin{equation*}\label{Vn in Vn+1}
V_{N+1}=V_N\oplus \widehat{W}_N,
\end{equation*} 
and flags in $\widehat{W}_N$
\begin{equation*}\label{flag for n}
W_{N,1}\subset...\subset
W_{N,k_N}\subset\widehat{W}_N,
\end{equation*}
such that each $\phi_N$ is given by:
\begin{equation*}\label{descr of phi n's}
\phi_N(0\subset V_{m_{N,1}}\subset...
\subset V_{m_{N,k_N}}\subset V_N)=\\
(0\subset V_{m_{N,1}}\oplus W_{N,1}\subset ...\subset V_{m_{N,k_N}}\oplus W_{N,k_N}\subset
V_{N+1}).
\end{equation*}
Set
\begin{equation*}\label{V=lim Vn}
V:=\underset{\to}\lim V_N. 
\end{equation*} 

Our aim is to define a basis $E$ of $V$ and a generalized flag
$\mathcal{\underline{F}}$ compatible with $E$,  so that the
direct limit of the strict standard extensions $\phi_N$ can be
identified with $\mathbf{Fl}(\mathcal{\underline{F}},E,V)$. 
Fix a flag $F_1=(0\subset V_{1,1}\subset...\subset V_{1,k_1} 
\subset V_1)\in Fl(m_{1,1},...,m_{1,k_1},V_1)$. Choose a basis 
\begin{equation*}\label{basis E}
E=\{e_{\alpha}\}_{\alpha\in\mathbb{Z}_{>0}}
\end{equation*} 
of $V$ such that, for all subspaces $T$ of $V$ of the form 
$V_{1,1},...,V_{1,k_1}$ and $W_{N,j}$ for $N$ and $j$, the 
set $T\cap E$ is a basis of $T$. Consider the following 
equivalence relation $\sim$ on the set $E$. We write 
\begin{equation*}\label{equiv}
e_{\alpha}\sim e_{\tilde{\alpha}}
\end{equation*}
if there exists $N_{\alpha}\in\mathbb{Z}_{>0}$ such that, for 
any $N\ge N_{\alpha}$, there is no space of the flag 
$\phi_N\circ\phi_{N-1}\circ...\circ\phi_1(F_1)$ containing 
$e_{\alpha}$ but not $e_{\tilde{\alpha}}$, or vice versa.
Using the fact that all embeddings $\phi_N$ are strict 
standard extensions, one checks that $\sim$ is an equivalence 
relation. Denote by $[e_{\alpha}]$ the equivalence class of 
the vector $e_{\alpha}$. 

Next, we claim that, by construction, 
the set $A$ of equivalence classes $[e_{\alpha}]$ is linearly 
ordered, and we will denote this linear ordering by the symbol 
$\prec$. Indeed, let $[e_{\alpha}]\ne[e_{\beta}]$. For 
$n\ge\max\{N_{\alpha},N_{\beta}\}$, consider the flag 
$\phi_N\circ\phi_{N-1}\circ...\circ\phi_1(F_1)$
and take its smallest subspaces containing respectively 
$e_{\alpha}$ and $e_{\beta}$. Since $[e_{\alpha}]\ne
[e_{\beta}]$, it follows that these spaces are not equal. 
By definition, we have $[e_{\alpha}]\prec[e_{\beta}]$ if the 
smallest space of the flag $\phi_N\circ\phi_{N-1}\circ...\circ
\phi_1(F_1)$ containing $e_{\alpha}$ is smaller than the 
smallest space of the same flag containing $e_{\beta}$. 

Finally, we define a generalized flag 
$\mathcal{\underline{F}}$, compatible with the 
basis $E$, and determined by the above order on $E$. For this, 
we associate two subspaces of $V$ to any equivalence class 
$a=[e_{\alpha}]$ :
\begin{equation}\label{flags F',F''}
F'_a=\mathrm{Span}\{e_{\beta}\ |\ [e_{\beta}]\prec a\}, \ \ \ 
\ \ F''_a=\mathrm{Span}\{e_{\beta}\ |\ [e_{\beta}]\preccurlyeq 
a\}.
\end{equation}
Then the set of vector subspaces of $V$
\begin{equation}\label{flag cal F0}
\mathcal{\underline{F}}=\{F'_a,F''_a\}_{a\in A}
\end{equation}
is easily seen to be a generalized flag in $V$ compatible with 
$E$. 

If, instead of \eqref{phi_n's}, we consider standard extensions
\begin{equation}\label{Z}
\psi_N:\ FlO(m_{N,1},...,m_{N,k_N},V_N)\hookrightarrow
FlO(m_{N+1,1},...,m_{N+1,k_{N+1}},V_{N+1})
\end{equation}
or
\begin{equation}\label{T}
\psi_N:\ FlS(m_{N,1},...,m_{N,k_N},V_N)\hookrightarrow
FlS(m_{N+1,1},...,m_{N+1,k_{N+1}},V_{N+1}),
\end{equation}
a similar construction of a relevant basis $E$ goes through.
First of all, in the case of \eqref{Z}, for our purposes it
suffices to assume that that the dimension of all spaces $V_N$
are simultaneously odd or even. We require $E$ to have the 
form $\{e_n,e_0,e^n\}_{n\in\mathbb{Z}_{>0}}$ in the odd case, 
and the form $\{e_n,e^n\}_{n\in\mathbb{Z}_{>0}}$ in the even 
case.
This latter form applies also to the case of \eqref{T}. In all
cases, $E$ has to be chosen by the same condition that all
subspaces of the form $V_{1,1},...,V_{1,k_1}$ and $W_{N,k_j}$
for $N\in\mathbb{Z}_{>0}$ are generated by subsets of $E$.
Next, in order to define a linear order on $E$, one applies to
the vectors $e_n$ the procedure outlined above, and then sets
$e^k\prec e^l\ \Leftrightarrow\ e_l\prec e_k$. Finally,
whenever there is a vector $e_0$ one puts $e_n\prec e_0\prec
e^k$ for any $k,n\in\mathbb{Z}_{>0}$. Then the generalized 
flag $\mathcal{\underline{F}}$ determined by formulas 
\eqref{flags F',F''} and \eqref{flag cal F0} is isotropic (in
the sense of the definition of the beginning of this section)
and an ind-variety $\mathbf{FlO}(\mathcal{\underline{F}},E,V)$ , or respectively $\mathbf{FlS}(\mathcal{\underline{F}},E,V)$
is well defined.

We are now ready for the following theorem.
\begin{theorem}\label{main thm}
There is an isomorphism of ind-varieties
$$
\varinjlim Fl(m_{N,1},...,m_{N,k_N},V_N)\simeq
\mathbf{Fl}(\mathcal{\underline{F}},E,V).
$$
Similarly, in the orthogonal and symplectic cases, there are
isomorphisms of ind-varieties
$$
\varinjlim FlO(m_{N,1},...,m_{N,k_N},V_N)\simeq
\mathbf{Fl}(\mathcal{\underline{F}},E,V),
$$
$$
\varinjlim FlS(m_{N,1},...,m_{N,k_N},V_N)\simeq
\mathbf{Fl}(\mathcal{\underline{F}},E,V).
$$
\end{theorem}
\begin{proof}
We consider only the case of ordinary flag varieties, and 
leave the other cases to the reader. Note 
first that $(m_{N,1},...,m_{N,k_N})$ is the type of the flag
$\mathcal{\underline{F}}\cap V_N$, so that $\mathbf{Fl}
(\mathcal{\underline{F}},E,V)=\varinjlim Fl(m_{N,1},...,
m_{N,k_N},V_N)$ where the direct limit is taken with respect 
to the embeddings 
$$
i_{\dim V_{N+1}-1}\circ...\circ i_{\dim V_N}:\ Fl(m_{N,1},...,
m_{N,k_N},V_N)\hookrightarrow Fl(m_{N+1,1},...,
m_{N+1,k_{N+1}},V_{N+1}).
$$
The embeddings $i_n$ were introduced in the first part of this
section, and are given by formulas \eqref{example of st ext} 
and \eqref{example of st ext2}, respectively.

However, we claim that our fixed standard extension $\phi_N$ 
equals the composition $i_{\dim V_{N+1}-1}\circ...\circ 
i_{\dim V_N}$. This follows from an iterated application of
Lemma \ref{Lemma 3.7} to the decompositions
$$
V_{N+1}=V'_{\dim V_{N+1}-1}\oplus\mathrm{Span}\{e_{\dim V_{N+1}}
\},
$$
$$
\ V'_{\dim V_{N+1}-1}=V'_{\dim V_{N+1}-2}\oplus\mathrm{Span}
\{e_{\dim V_{N+1}-1}\},...,
$$
$$V'_{\dim V_N+1}=V_N\oplus\mathrm{Span}\{e_{\dim V_N+1}\},
$$ 
and from the observation that the corresponding standard
extensions
$$
Fl(m_{n,1},...,m_{n,k_n},V'_n)\hookrightarrow Fl(m_{n+1,1},...,
m_{n+1,k_{n+1}},V'_{n+1})
$$
arising in this way, are determined simply by the splitting
$V'_{n+1}=V'_n\oplus\mathrm{Span}\{e_{n+1}\}$. Since the 
standard extension $i_n$ is determined by the same 
decomposition, the statement follows. 
\end{proof}

The following corollary can be considered as the main result 
of this paper.

\begin{corollary}\label{cor 5.3}
The direct limit of any admissible sequence of linear 
embeddings, $\varinjlim Fl(m_{N,1},...m_{N,k_N},V_N)$, 
$\varinjlim FlO(m_{N,1},...m_{N,k_N},V_N)$, or $\varinjlim 
FlS(m_{N,1},...m_{N,k_N},V_N)$, is a homogeneous ind-variety
for the group $SL(\infty),\ O(\infty)$ or $Sp(\infty)$, 
respectively.
\end{corollary}

The claim of Corollary \ref{cor 5.3} can be derived more 
directly from Corollary \ref{Cor 4.4} by showing that any 
direct limit of standard extensions is a homogeneous 
ind-variety, but Theorem \ref{main thm} provides an explicit 
description of such a direct limit as an appropriate 
ind-variety of generalized flags. We should also point out that
homogeneous ind-varieties of the ind-groups $GL(\infty),$ $SL(\infty),$ $O(\infty),$ $Sp(\infty)$ have been studied in papers preceding \cite{DP}, see \cite{DPW} and the references
therein.

\vspace{1cm}
\section{Appendix}\label{special}
\vspace{5mm}

In this appendix, we construct ind-varieties which are not 
isomorphic to ind-varieties of generalized flags, but 
nevertheless are direct limits of linear embeddings of flag 
varieties. Here we use the notation $\mathbb{P}(V)$ also for a
countable-dimensional vector space. $\mathbb{P}(V)$ is the 
ind-variety of 1-dimensional subspaces of $V$. We also write 
$\mathbb{P}^{\infty}$ instead of $\mathbb{P}(V)$ when we do 
not need to specify $V$.

First, consider the following chain of linear embeddings
$$
...\hookrightarrow Fl(1,2^n-1,V_n)\overset{k_n}
{\hookrightarrow}G(1,V_n)\times G(2^n-1,V_n)\overset{j_n}
{\hookrightarrow}Fl(1,2^{n+1}-1,V_n\oplus V_n)\overset{k_{n+1}}
{\hookrightarrow}
$$
$$
\overset{k_{n+1}}{\hookrightarrow}G(1,V_n\oplus V_n)
\times G(2^{n+1}-1,V_n\oplus V_n)\hookrightarrow...\ ,
$$
where $\dim V_n=2^n,\ k_n$ and $k_{n+1}$ are the canonical 
embeddings, and $j_n(V_1,V_{2^n-1})=(V_1\subset V\oplus0\subset V\oplus V_{2^n-1})$ for subspaces $V_1,\ 
V_{2^n-1}\subset V$ of respective dimensions 1 and $2^n-1$. Clearly, the embedding
$$
j_n\circ k_n:Fl(1,2^n-1,V_n)\hookrightarrow Fl(1,2^{n+1}-1,V_n)
$$
is linear but does not satisfy condition (b) of Theorem 
\ref{Thm 4.3} as it factors through the embedding $k_n$. The 
direct limit $\varinjlim Fl(1,2^n-1,V_n)$ is isomorphic as an 
ind-variety to the direct limit of embeddings
$$
G(1,V_n)\times G(2^n-1,V_n)\overset{k_{n+1}\circ j_n}
{\hookrightarrow}G(1,V_n\oplus V_n)\times G(2^n-1,V_n\oplus 
V_n),
$$
which is easily checked to be isomorphic to the direct product
$\mathbb{P}(V)\times\mathbb{P}(V)$ for a countable-dimensional
vector space $V$. The ind-variety $\mathbb{P}(V)\times
\mathbb{P}(V)$ is not isomorphic to an ind-variety of 
generalized flags.

Next, we will give a more interesting example in which 
condition (c) is not satisfied. More precisely, we will 
construct a linear embedding $\phi: Fl(m_1,m_2,V)
\hookrightarrow Fl(n_1,n_2,V')$ that will have the property 
that $p(1)=1,\ p(2)=2,$ $\phi_{2,2}:G(m_2,V)\to G(n_2,V')$ is 
a standard extension, but $\phi_{1,1}:G(m_1,V)\to G(n_1,V')$ 
factors through a projective subspace of $G(n_1,V')$.

Let $3\le\dim V<\infty$, fix positive integers $m_1,\ m_2,\ 
1< m_1< m_2<\dim V,$ and let $V^0$ be a subspace of $V$ of 
dimension $\dim V-m_1+1$. Consider the rational morphism
\begin{equation*}\label{linear prn p}
\gamma:G(m_1,V)\dasharrow\mathbb{P}(V^0),\ V_{m_1}\mapsto  
V_{m_1}\cap V^0.
\end{equation*}
Assume $G(m_1,V)$ is embedded into $\mathbb{P}(\wedge^{m_1}V)$ 
via the Pl\"ucker embedding, and let $Y:=\{V_{m_1}\in G(m_1,V
)\ |\ \dim(V_{m_1}\cap V^0)\ge2\}$. A standard computation in 
linear algebra shows that\\ 
(i) there exists a subspace $W\subset\wedge^{m_1}V$ of 
codimension $\dim V-m_1+1$, such that 
$Y=G(m_1,V)\cap \mathbb{P}(W)$;\\ 
(ii) there is an isomorphism $g:(\wedge^{m_1}V)/W
\xrightarrow{\simeq}V^0$ satisfying 
\begin{equation}\label{p vs g}
\gamma(V_{m_1})=g(\wedge^{m_1}V_{m_1}+W)
\end{equation}
(in particular, this implies that $\gamma$ is regular on 
$G(m_1,V)\setminus Y$);\\
(iii) there exists a vector space $U$ containing $\wedge^{m_1}
V$ as a subspace, together with a surjective operator 
$\varepsilon:U\twoheadrightarrow V$ with $\ker\varepsilon=W$.

In addition, we may suppose that $m_1$ is 
large enough so that there exists a subspace $Z$ of $W$ such 
that the morphism $\phi':\ G(m_1,V)\to \mathbb{P}((\wedge^{m_1}
V)/Z),\ V_{m_1}\mapsto\wedge^{m_1}V_{m_1}+Z$ is an embedding.
Set $V':=U$, $n_1=\dim Z+1,\ n_2=\dim W+m_2$. The inclusion
$\wedge^{m_1}V\subset V'$ yields an embedding $j:\mathbb{P}((
\wedge^{m_1}V)/Z)\hookrightarrow G(n_1,V'),\ v+Z\mapsto
\mathrm{Span}\{v+Z\}$. Define $\varphi_{1,1}: G(m_1,V)\to 
G(n_1,V')$ as the composition $j\circ\phi'$, and let 
$\varphi_{2,2}:G(m_2,V)\to G(n_2,V')$ be the standard 
extension defined by the flag $(W\subset U)$. 

We show now that, given a flag $(0\subset V_{m_1}\subset 
V_{m_2}\subset V)$, one has 
$\varphi_{1,1}(V_{m_1})\subset\varphi_{2,2}(V_{m_2})$, and 
hence there is a well-defined embedding 
$$\varphi:\ Fl(m_1,m_2,V)\hookrightarrow Fl(n_1,n_2,V'),\ 
(V_{m_1}\subset V_{m_2})\mapsto(\varphi_{1,1}(V_{m_1})\subset
\varphi_{2,2}(V_{m_2})).
$$ 
Indeed, in view of \eqref{p vs g}, the rational morphism 
$\gamma$ decomposes as
$$
\gamma:\ G(m_1,V)\overset{\phi'}{\hookrightarrow}\mathbb{P}
((\wedge^{m_1}V)/Z)\overset{q}{\dasharrow}\mathbb{P}(
(\wedge^{m_1}V)/W)\xrightarrow[\simeq]{g}\mathbb{P}(V^0),\\
$$
$$
V_{m_1}\overset{\phi'}{\mapsto}\wedge^{m_1}V_{m_1}+Z\overset
{q}{\mapsto}\wedge^{m_1}V_{m_1}+W\overset{g}
{\mapsto} V_{m_1}\cap V^0, 
$$
where $q$ is a rational surjective morphism. If 
$V_{m_1}\cap V^0=:V_1$ is a 1-dimensional space, 
i.e. if $q$ is regular at the point $\wedge^{m_1}V_{m_1}+Z
\in\mathbb{P}((\wedge^{m_1}V)/Z)$, then the inclusion 
$V_{m_1}\subset V_{m_2}$ implies $V_1\subset 
V_{m_2}$. Hence, $\phi_{1,1}(V_{m_1})=\wedge^{m_1}V_{m_1}+Z
\subset\wedge^{m_1}V_{m_1}+W=\varepsilon^{-1}(V_1)\subset 
\varepsilon^{-1}(V_{m_2})=\phi_{2,2}(V_{m_2})$. In the 
remaining case when $\dim(V_{m_1}\cap V^0)\ge2$, we have 
$\wedge^{m_1}V_{m_1}\subset W$ by property (i), and therefore
$\phi_{1,1}(V_{m_1})=\wedge^{m_1}V_{m_1}+Z\subset W\subset 
\varepsilon^{-1}(V_{m_2})=\phi_{2,2}(V_{m_2})$.

Finally, we have the following proposition.
\begin{proposition}
Let $\{\phi_k:\ Fl(m_{k,1},m_{k,2},V_k)\to Fl(m_{k+1,1},m_{k+1,
2},V_{k+1})\}_{k\ge1}$ be a chain of embeddings as constructed 
above. The ind-variety $\mathbf{X}$ obtained as the direct 
limit of this chain is not isomorphic to an ind-variety of 
generalized flags.
\end{proposition}

\begin{proof}
Assume to the contrary that $\mathbf{X}$ is isomorphic to 
$\mathbf{Y}$ for some ind-variety of generalized flags 
$\mathbf{Y}$. Since the embeddings $\phi_k$ are linear, it 
follows that $\mathrm{Pic}\mathbf{X}\simeq\mathbb{Z}\times
\mathbb{Z}$. Therefore $\mathrm{Pic}\mathbf{Y}\simeq\mathbb{Z}
\times\mathbb{Z}$, and consequently, $\mathbf{Y}$ is 
isomorphic to $\mathbf{Fl}(F',E',V')$ for some
countable-dimensional vector space $V'$, some basis $E'$ of 
$V'$, and some flag $F'=(F'_1\subset F'_2)$ in $V'$ of length 
2. Since the 
morphisms $(\phi_k)_{1,1}:\ \ G(m_{k,1},V_k)\to G(m_{k+1,1},
V_{k+1})$ factor through projective spaces, the ind-variety 
$\mathbf{X}$ projects onto $\mathbb{P}^{\infty}$ in a way that 
the line bundle $\mathcal{O}_{\mathbf{X}}(1,0)$ is trivial 
along the fibers of the projection. Therefore, we infer that
$\dim F'_1=1$ or $\mathrm{codim}_{V'}F'_2=1$. This follows 
from the fact that the ind-variety $\mathbb{P}^{\infty}$ is not
isomorphic to any ind-grassmannian $\mathbf{Fl}(F,E',V')$, 
where $F$ is a single subspace with $\dim F\ge2$ and 
$\mathrm{codim}_{V'}F'\ne1$, see \cite[Thm. 2]{PT}. 
Consequently, the flag $F'=(F'_1\subset F'_2)$ can be chosen 
with $\dim F'_1=1$ (in the case where $\mathrm{codim}_{V'}F'_2
=1$ one replaces $V'$ by its restricted dual space defined by 
the basis $E'$). 

The standard extensions $(\phi_k)_{2,2}:G(m_{k,2},V_k)\to G(
m_{k+1,2},V_{k+1})$ allow to identify $\varinjlim G(m_{k,2},V
_k)$ with an ind-grassmannian $\mathbf{Fl}(F_{\infty},E,V)$, 
where $F_{\infty}$ is a subspace of $V=\varinjlim V_k$ and $E$ 
is an appropriate basis of $V$. Moreover, we have $\dim 
F_{\infty}=\infty=\mathrm{codim}_{V}F_{\infty}$, as the 
construction of $\phi_k$ shows that $\lim\limits_{k\to\infty}
m_{k,2}=\infty=\lim\limits_{k\to\infty}(\dim V_k-m_{k,2})$. 
After identifying the triples $(F_{\infty},E,V)$ and $(F'_2,E',
V')$, we obtain a commutative diagram
$$
\xymatrix{
\mathbf{X}\ar[dr]_-{\pi_{\mathbf{X}}}& 
\mathbf{Fl}(F,E,V)\ar[d]^-{\pi}\ar[l]^-{\sim}_-{\sigma}\\
& \mathbf{Fl}({F_{\infty}},E,V),} 
$$
where $\pi$ is the natural projection and $\sigma$ is an 
isomorphism of ind-varieties. The fibers of both projections 
$\pi_{\mathbf{X}}$ and $\pi$ are isomorphic to $\mathbb{P}
^{\infty}$.

We will show now that the existence of the isomorphism 
$\mathbf{X}\xleftarrow[\sim]{\sigma}\mathbf{Fl}(F,E,V)$ is 
contradictory. Recall that the group $GL(E,V)$ of invertible finitary linear operators defined by $E$ (i.e. the group of 
invertible linear generators on $V$ each of which fixes all but finitely many elements of $E$) acts on $\mathbf{Fl}(F,E,V)
$ and $\mathbf{Fl}({F_{\infty}},E,V)$, and the line bundle 
$\mathcal{O}(1,0):=\sigma^*\mathcal{O}_{\mathbf{X}}(1,0)$ on 
$\mathbf{Fl}(F,E,V)$ admits a $GL(E,V)$-linearization. This 
linearization is unique when restricted to $SL(E,V)$. If we 
compute the $SL(E,V)$-module $\Gamma:=H^0(\mathbf{Fl}(F,E,V),
\mathcal{O}(1,0))$, we see that $\Gamma\simeq\varprojlim 
H^0(\pi_{k*}(\mathcal{O}(1,0)|_{Fl(1,m_{k,2},V_k)}))$, 
where here $\pi_k: Fl(1,m_{k,2},V_k)\to Gr(m_{k,2},V_k)$ 
denote the natural projections. Consequently,
$$
\Gamma\simeq\varprojlim V_k^*\simeq V^*.
$$

On the other hand, since $\sigma^*$ induces an 
$SL(E,V)$-linearization on $\mathcal{O}_{\mathbf{X}}(1,0)$, and
consequently an isomorphism of $SL(E,V)$-modules
$\Gamma\xrightarrow{\sim}H^0(\mathbf{X},
\mathcal{O}_{\mathbf{X}}(1,0))$, we can compute $\Gamma$ via 
the system of projections $\tau_k: Fl(m_{k,1},m_{k,2},V_k)\to 
G(m_{k,2},V_k)$. This yields
$$
\Gamma\simeq\varprojlim 
H^0(\tau_{k*}(\mathcal{O}_{\mathbf{X}}(1,0)|_{Fl(m_{k,1},
m_{k,2},V_k)}))\simeq\varprojlim\wedge^{m_k}V_k^*.
$$
However, $\varprojlim\wedge^{m_k} V_k^*$ is not isomorphic to
$V^*$ as an $SL(E,V)$-module. To see this, it is enough to 
observe that $\varprojlim\wedge^{m_k} V_k^*$ and $V^*$ are
non-isomorphic after restriction to $SL(V_k)$ for large $k$.
We have a contradiction as desired.
\end{proof}

\end{document}